\documentclass[11pt]{article}

\makeatletter
\DeclareRobustCommand{\qed}{%
  \ifmmode 
  \else \leavevmode\unskip\penalty9999 \hbox{}\nobreak\hfill
  \fi
  \quad\hbox{\qedsymbol}}
\newcommand{\openbox}{\leavevmode
  \hbox to.77778em{%
  \hfil\vrule
  \vbox to.675em{\hrule width.6em\vfil\hrule}%
  \vrule\hfil}}
\newcommand{\qedsymbol}{\openbox}
\newenvironment{proof}[1][\proofname]{\par
  \normalfont
  \topsep6\p@\@plus6\p@ \trivlist
  \item[\hskip\labelsep\itshape
    #1.]\ignorespaces
}{%
  \qed\endtrivlist
}
\newcommand{\proofname}{Proof}
\makeatother

\usepackage[utf8]{inputenc}

\usepackage{color}
\usepackage{latexsym}
\usepackage{dsfont}
\usepackage{amssymb}
\usepackage{comment}
\usepackage{graphicx}
\usepackage{amsmath,amsfonts,amssymb,theorem,euscript,array,enumerate,amsfonts,mathrsfs}
\usepackage{hyperref}
\usepackage{appendix}
\usepackage[T1]{fontenc}
\usepackage{babel}
\numberwithin{equation}{section}
\usepackage{bbm}
\usepackage{subfigure}
\usepackage{color}

\usepackage{stmaryrd}

\usepackage[algo2e,ruled,vlined]{algorithm2e} 
 \usepackage{algorithm}
 \usepackage{algorithmicx}
\usepackage{algpseudocode}
\usepackage{ marvosym }
\usepackage{csquotes}
\usepackage{hyperref}

\usepackage{mathtools}
\mathtoolsset{showonlyrefs}

\usepackage{comment}

\usepackage{tikz}
\usetikzlibrary{fit,matrix,chains,positioning,decorations.pathreplacing,arrows}

\usepackage{mathtools}
\mathtoolsset{showonlyrefs}

\usepackage[a4paper,margin=1in,heightrounded]{geometry}

\usepackage{algpseudocode}
\usepackage{algorithm}

\usepackage[normalem]{ulem}

\def \b1{\bf{1}}

\def \I{\mathbb{I}}
\def \N{\mathbb{N}}
\def \R{\mathbb{R}}

\def \E{\mathbb{E}}
\def \F{\mathbb{F}}

\def \P{\mathbb{P}}
\def \Q{\mathbb{Q}}

\def \bd{\mathbb{d}}

\def \bd{\boldsymbol{d}} 
\def \bW{\boldsymbol{W}}

\def \mfs{\mathfrak{\sigma}} 
 
\def \mfb{\mathfrak{b}}

\def \d{\mathrm{d}}

\def \bomu{\boldsymbol{\mu}}
\def \bonu{\boldsymbol{\nu}}
\def \boeta{\boldsymbol{\eta}}
\def \bodelta{\boldsymbol{\delta}}

\def\esssup_#1{\underset{#1}{\mathrm{ess\,sup\, }}}

\def\argmin_#1{\underset{#1}{\mathrm{argmin\, }}}
\def\argmax_#1{\underset{#1}{\mathrm{argmax\, }}}

\def \Bc{{\cal B}}
\def \Cc{{\cal C}}

\def \Ec{{\cal E}}
\def \Fc{{\cal F}}
\def \Gc{{\cal G}}

\def \Ic{{\cal I}}

\def \Lc{{\cal L}}
\def \Pc{{\cal P}}

\def \Mc{{\cal M}}
\def \Nc{{\cal N}}

\def \Wc{{\cal W}}

\def\boX{{\boldsymbol X}}

\def\bolx{{\boldsymbol x}}

\def\bomu{{\boldsymbol \mu}}

\def\bozeta{{\boldsymbol \zeta}}

\def \mfb{\mathfrak{b}}

\def \d{\mathrm{d}} 
\def \mfs{\mathfrak{s}}

\def\beqs{\begin{eqnarray*}}
\def\enqs{\end{eqnarray*}}
\def\beq{\begin{eqnarray}}
\def\enq{\end{eqnarray}}

\def\blu#1{{\color{blue}#1}}

\newtheorem{Theorem}{Theorem}[section]
\newtheorem{Definition}{Definition}[section]
\newtheorem{Proposition}{Proposition}[section]
\newtheorem{Assumption}{Assumption}[section]
\newtheorem{Lemma}{Lemma}[section]

\newtheorem{Remark}{Remark}[section]

\numberwithin{equation}{section}

\title{
Nonlinear Graphon mean-field systems
}

\author{Fabio Coppini\footnote{Utrecht University, The Netherlands \sf f.coppini at uu.nl}  \and  Anna De Crescenzo\footnote{LPSM, Universit\'e Paris Cit\'e and Sorbonne University, \sf decrescenzo at lpsm.paris} \and  
Huy\^en Pham\footnote{LPSM, Universit\'e Paris Cit\'e and Sorbonne University, \sf pham at lpsm.paris. 
The work of this author is  partially supported by the BNP-PAR Chair ``Futures of Quantitative Finance", and the 
Chair Finance \& Sustainable Development / the FiME Lab (Institut Europlace de Finance)}}

\date{}

\begin{document}

\maketitle

\begin{abstract}
    We address a system of weakly interacting particles where the heterogenous connections among the particles are described by a graph sequence and the number of particles grows to infinity. Our results extend the existing law of large numbers and propagation of chaos results to the case where the interaction between one particle and its neighbors is expressed as a nonlinear function of the local empirical measure. In the limit of the number of particles which tends to infinity, if the graph sequence converges to a graphon, then we show that the limit system is described by an infinite collection of processes and can be seen as a process in a suitable $L^2$ space constructed via a Fubini extension. The proof is built on decoupling techniques and careful estimates of the Wasserstein distance.
\end{abstract}



\vspace{3mm}


\vspace{3mm}


\section{Introduction}

\paragraph{Motivation for complex heterogenous systems.}
Physicists and mathematicians have always been fascinated by modeling complex systems with applications that include evolutionary biology \cite{arenas_synchronization_2008, agathe-nerine_long-term_2022}, epidemiology \cite{keliger_local-density_2020, delmas_infinite-dimensional_2020}, game theory and controls \cite{delarue_mean_2017, caines_graphon_2020} and economics \cite{parise_graphon_2020}. One way to represent this complexity is through a mixture of nonlinear dynamics and non-trivial connections among the particles. In a finite system of interacting particles, it is natural to introduce a graph object describing how each particle is connected to the others: it suffices to multiply the interaction between two particles by the indicator function of the edge connecting the corresponding vertices in the underlying graph. Even if this mathematical representation is simple, it does not say anything about the macroscopic limit of the population, i.e., is the empirical measure converging as the number of particles goes to infinity? How is the graph structure influencing the dynamics? To address these questions several results have been proposed, they usually focus on the study of the empirical measure, or on the particle trajectories, and mainly prove Propagation of Chaos \cite{delattre_note_2016, bhamidi_weakly_2019}, Law of Large Numbers \cite{oliveira_interacting_2019, coppini_law_2020, bayraktar_graphon_2023, bet_weakly_2023}, Central Limit Theorem \cite{bhamidi_weakly_2019, coppini_central_2022} and Large Deviation Principle \cite{oliveira_interacting_2019, coppini_law_2020, maclaurin_large_2020}. We refer to the literature section below for more references.

While  the classical mean-field systems are studied in great generality with respect to the particle interaction, the existing results on particle systems and graph sequences focus on the linear (or scalar) case, i.e., the interaction between the particles is one to one and the sum of the interactions on a single particle is equivalent to depending linearly on the empirical measure. To the authors' knowledge, the gap between mean-field systems and interacting particles on graph sequences has never been tackled in the literature and demands to be clarified, notably in the application viewpoint where the linear dependence is a strong requirement.
In the case where the particles are influenced by a random noise, the It\^o formula \cite{oksendal_stochastic_2003} is one of the most important tools used to study such systems. Notably, it can be used to derive the partial differential equation, usually called Vlasov or McKeanVlasov equation, satisfied by the empirical measure as the number of particles goes to infinity. For interacting particle systems, it is well known that the limit equation is nonlinear and it is defined on a suitable space containing the probability measures \cite{gartner_mckean-vlasov_1988, coghi_pathwise_2020}: in many instances, as in the case of control problems \cite{cosetal21}  it is important to have an $L^2$-formulation. To the authors' knowledge, in the literature there is no such formulation whenever the limit equation depends on the graph structure: an Hilbertian formulation would benefit the mathematical community as it allows the application of the It\^o formula.

\paragraph{Contributions of this work.}
The aim of this work is twofold: (1) to extend existing Law of Large Numbers (LLN) and a Propagation of Chaos (POC) for interacting particle systems on graphs to the case where the interaction is Lipschitz with respect to the empirical measure; (2) to define an infinite system of nonlinear McKean-Vlasov equations in a suitable $L^2$-space, such that the corresponding non-linear processes are measurable (in some sense to be defined). We remark that, even taken independently, each of these points has not been addressed in the literature: the closest result to (1) is given by \cite{bayraktar_graphon_2023} where the dependence with respect to the empirical measure is linear (see Remark \ref{rem:comparison} for a comparison of the particle systems); for (2), \cite{aurell_stochastic_2022} studies linear-quadratic stochastic differential games and makes use of a Fubini extension to carefully define the system in the infinite particle setting, however the dynamics does not involve probability measures. We refer to the next subsection for a review of existing results.

Our results can be established thanks to the advances already made in the case of graphon particle systems \cite{bayraktar_graphon_2023, bet_weakly_2023}, the ones involving a suitable extension of the Lebesgue measure to address an uncountable family of stochastic processes \cite{aurell_stochastic_2022, sun_exact_2006} and known results on the Wasserstein distance \cite{fournier-guillin, villani_optimal_2009}. The proof of LLN, see Theorem \ref{theoLLN}, and the one of POC, see Theorem \ref{theoPOC}, mimic the classical one given by Sznitman \cite{sznitman_topics_1991} where the finite particle system is paired with a finite sample of a suitable limit system. We recall that this last sample of random variables must have suitable exchangeability properties. In our case, given the underlying graphon structure, the trajectories are still independent of each other, but they are no longer identically distributed. The existence and uniqueness of solutions for the infinite system is obtained with a classical fixed point argument, but now in a suitable $L^2$ space constructed using a Fubini extension, see Definition \ref{d:solution-graphonsysgen} and Proposition \ref{p:existence-uniqueness-graphonsysgen}. The main difficulty coming from the assumption on the dynamics, is that one has to uniformly control all the local empirical measures: by coupling the trajectories with the ones of the infinite system, this translates into having a uniform estimate on the local empirical measures of the infinite system. This is possible by extending a well-known result by Fournier and Guillin \cite{fournier-guillin} on the convergence of empirical measures in the Wasserstein distance to a not identically distributed setting, see Lemma \ref{lemFG} for a precise statement. We are able to recover the same convergence rate, which depends on the dimension of the space, and thus obtain the expected convergence rate in the POC result. We choose to focus on deterministic dense graph sequences converging to graphons \cite{lovasz_large_2012}, instead of more general graph sequences \cite{bayraktar_graphon_2023, jabin_mean-field_2022}. However, we believe that with a bit more work and notational effort, our result can be extended to random graph sequences with diverging average degree that converge to graphon under a suitable renormalisation, we refer to \cite{oliveira_interacting_2020, coppini_law_2020, bayraktar_graphon_2023} for some results applying to dense and not-so-dense graph sequences. We also refrain from investigating singular interactions \cite{lucon_mean_2014, wang_mean-field_2022}, but they may represent a subject for further study.

\paragraph{Existing literature.}
Beginning with the works \cite{delattre_note_2016, bhamidi_weakly_2019}, numerous papers have discussed interacting particle systems on graph sequences, see \cite{coppini_law_2020, oliveira_interacting_2019, lucon_quenched_2020, bayraktar_graphon_2023} this list being in no way exhaustive. It is now well-known that, for a general dense graph sequence, the empirical measure of a system of interacting particles converges to the solution of a partial differential equation, usually called Vlasov or McKean-Vlasov equation, that depends on the graph limit itself, see \cite{coppini_note_2022} for more discussion on the equation and its relation with the graph limit. 
Depending on the assumptions on the interaction, on the initial conditions or on the graph sequence itself, several results have been established. It should be mentioned that in the literature the interaction on a given particle is always assumed to be the sum of the two point particle interactions, i.e., linear with respect to the local empirical measures. When the two-point particle interaction is Lipschitz and the underlying graph sequence is somewhat homogeneous, e.g., Erd\H{o}sh-Rényi random graphs, classical trajectorial estimates \cite{delattre_note_2016, bhamidi_weakly_2019, coppini_law_2020} lead to a Law of Large Numbers and propagation of chaos results; these results show that the limit equation is the same of the classical mean-field case \cite{sznitman_topics_1991, oelschlager_martingale_1984, braun_vlasov_1977}. For particle systems where the underlying graph displays some sort of heterogeneity, several methods have been proposed.

The graphon theory \cite{lovasz_limits_2006, lovasz_large_2012} has found its place in several works \cite{oliveira_interacting_2019,bayraktar_concentration_2022, bayraktar_graphon_2023, bayraktar_stationarity_2022, bet_weakly_2023}, notably in the case of graphon mean field systems \cite{parise_graphon_2020, caines_graphon_2020, lacker_label-state_2023}. 
Today, it represents one of the most cited framework when dense graph sequences are in play. The first article to link the cut-distance, the natural distance of the graphon space (see Section \ref{sec:graphon} below), to interacting particle systems is given by \cite{oliveira_interacting_2019}, where the authors consider spatially extended particles that are interacting through a random graph, this last one being sampled from a graphon. Similar results, although in a different metric concerning the graph convergence, are those \cite{lucon_mean_2014, lucon_quenched_2020} but also the ones about the Kuramoto model in the deterministic setting, see \cite{medvedev_continuum_2018, medvedev_nonlinear_2014, chiba_mean_2019} and the references therein. The works \cite{bayraktar_graphon_2023, bet_weakly_2023} have extended the work \cite{oliveira_interacting_2020} in two important ways: the first one \cite{bayraktar_graphon_2023} considers a large class of particle systems in $\R^d$ and shows that, for a graph sequence converging in cut-distance to a deterministic graphon, a Law of Large Numbers holds for the empirical measure and its limit is characterised by being the solution of an infinite system of partial differential equations; the second one \cite{bet_weakly_2023} restricts the analysis to particle systems on the one-dimensional torus but under weaker assumptions on the graph convergence, which can be in probability and where the limit graphon can be random itself. 
The paper \cite{amicaosul22} shows convergence of interacting mean-field particle system with inhomogeneous interactions to graphon mean-field BSDE, and in \cite{bayetal23b}, the authors prove a propagation of chaos for a system of forward-backward SDE with graphon interaction arising from  mean-field game problem. In both papers, the interaction is of linear form.
For the Central Limit Theorem, we refer to \cite{coppini_central_2022}, where the authors are able to study the global and local fluctuations of the empirical measure in the case of Erd\H{o}s-Rényi random graphs. Other relevant works concern the long time behavior, we refer to \cite{coppini_long_2022, bayraktar_stationarity_2022}. Singular interactions have been studied in \cite{lucon_mean_2014,wang_mean-field_2022}, the latter being  restricted to the case where the graph weights are on the vertices rather than the edges.

Graphon theory comes with some constraints, notably the fact that every vertex has roughly the same order of neighbors but also the symmetry assumed at the level of the limit graph (as a graphon is a symmetric function on the unit square). Two generalisations have been proposed: \cite{jabin_mean-field_2022} proves a Law of Large Numbers under some weak assumption on the degrees of the graph, \cite{kuehn_vlasov_2022} precisely focuses on sequences of directed graphs.

\paragraph{Outline of the paper.} The plan of the paper is organized as follows. In Section \ref{s:definition}, we define the interacting particle systems under consideration and state the main assumption on the particle dynamics. Section \ref{s:l2-formulation} is devoted to the $L^2$-formulation: it recalls the main results concerning the Fubini extension, the main tools and assumptions to work in the graphon setting and finally states the classical result on existence and uniqueness of solutions. The LLN and POC are stated in the Section \ref{s:lln+poc} together with their proofs. An Appendix contains some extra material used throughout the proofs.

\paragraph{Notations.}
We work in $\R^d$ where $d$ is some positive integer. The Euclidian norm of a vector in $\R^d$ is denoted by $\left|\cdot\right|$, and $\cdot$ is the scalar product. The space of probability measure over some metric space $(E, \bd)$ is denoted by $\Pc(E)$, and $\Mc_+(E)$ is the  set of nonnegative finite Borel  measures on $E$. 
When we consider the subspace of probability measures with finite variance, we denote it by $\Pc_2(E)$. Their elements are denoted by Greek letters, e.g., $\mu, \nu, \eta \in \Pc(E)$.

We equip $\Pc_2(E)$ with the 2-Wasserstein distance $\Wc_2$ so that it is a metric space itself. The 2-Wasserstein distance on $\Pc_2(E)$ is defined by, for $\mu, \nu \in \Pc_2(E)$ 
\begin{equation}\label{def:w_2-distance}
    \Wc_2(\mu, \nu) = \left( \inf_{X\sim\mu, Y\sim\nu} \E\Big[\bd(X, Y)^2\Big] \right)^{1/2}, 
\end{equation}
where $X, Y$ are random variables with law given respectively by $\mu$, and $\nu$.
The Kantorovitch duality, see  \cite[Theorem 5.10]{villani_optimal_2009}, states that for any $\mu,\nu\in\Pc_2(E)$

\begin{align}\label{eq:kantorovitch}
    \Wc_2^2(\nu,\mu) &= \sup\Bigg\{ \int_{E} f(y)\mu(\d y) + \int_{E}g(y)\nu(\d y), \quad \text{with } f,g\in C_b(E) \text{ such that } \\
& \hspace{3cm}  f(x)+g(y) \; \le \bd(x,y)^2 \quad \forall x,y\in E \Bigg\},
\end{align}
where $C_b(E)$ stands for the space of bounded continuous functions on $E$. In the same book by Villani, see \cite[Theorem 6.15]{villani_optimal_2009}, one can also find a control on the Wasserstein distance by means of the variation between (signed) measures: if $(E, \bd)$ is a Polish space with $x_0 \in E$, then it holds that for $\mu, \nu \in \Pc_2(E)$
\begin{align}
\label{eq:w-control-tv}
\Wc_2^2(\mu, \nu) & \le \;  2 \int_{\R^d} \bd(x, x_0)^2 \big| \mu - \nu \big|(\d x),
\end{align}
where $\big| \mu - \nu  \big|$ is the variation of the (signed) measure $\mu - \nu$. Recall that if $M$ is a signed measure, then $|M|=M^{+}+M^{-}$, where $M^+$ and $M^-$ are respectively the positive and the negative parts of $M$.

\section{Mean-field particle systems with graphon interactions}
\label{s:definition}

We consider a multi-agent/particle  system with heterogenous interactions where each agent interacts with all other agents via an aggregated mean-field effect of the whole population of size $N$. Denoting by $X_t^{i,N}$ 
the state valued in $\R^d$ at time $t>0$ of agent $i$ (represented by a vertex/node) 
$\in$ $\llbracket 1,N\rrbracket := \{1, \dots, N\}$,  the influence  of the other agents on  agent $i$   is given by the neighbourhood   empirical measure 
\beq
\label{d:local-emp-mes}
\nu_t^{i,N}  &=& \frac{1}{N_i} \sum_{j=1}^N \xi_{ij}^N \delta_{X_t^{j,N}}, 
\enq
where $(\xi_{ij}^N)_{i,j} \in \R^{N\times N}$ is the interaction matrix and $N_i$ is the number of edges of node  $i$, i.e. $N_i$ $=$ $\sum_{j=1}^N \xi_{ij}^N$, called degree of interaction. The interaction matrix $(\xi_{ij}^N)_{i,j}$ is associated to a graphon $G_N$, i.e., a  symmetric measurable function from $I\times I$ into $[0,1]$, with $I$ $:=$ $[0,1]$, via
\beqs
\xi_{ij}^N &=& G_N(u_i,u_j), 
\enqs  
where $u_i$ $=$ $i/N$ represents the label of agent $i \in \llbracket 1,N\rrbracket$, in the population of $N$ agents, and it is assumed that $G_N$ is a step graphon, i.e. 
\begin{align}
G_N(u,v) &= \; 
G_N\Big(\frac{\lceil Nu\rceil}{N},\frac{\lceil Nv\rceil}{N} \Big), \quad (u,v) \in I\times I.
\end{align}
We assume that there is no isolated particle, i.e., that for every $N$, $\inf_{i=1, \dots, N} N_i>0$, which  implies that
\beqs
 \| G_N(u,.)\|_{_1}  \;  := \; \int_{[0,1]} G_N(u,v) \d v > 0, \quad \forall u \in [0,1]. 
\enqs

The states of  the  $N$ agents/particles  are  then governed by the  particle system:
\begin{equation}\label{sysN} 
\d X_t^{i,N}  \; = \;  b(X_t^{i,N},\nu_t^{i,N}) \d t \; +  \;  \sigma(X_t^{i,N},\nu_t^{i,N}) \d W_t^{u_i}, \quad 0 \leq t \leq T, \;\;\;  i \in \llbracket 1,N\rrbracket,  
\end{equation}
where 
$\{W^u: u \in [0,1]\}$ is a collection of  i.i.d. $n$-dimensional Brownian motions on some filtered probability space $(\Omega,\Fc,\F=(\Fc_t)_t,\P)$, and the measurable coefficients $b$ $:$ $\R^d\times\Pc_2(\R^d)$ $\rightarrow$ $\R^d$, 
$\sigma$ $:$ $\R^d\times\Pc_2(\R^d)$ $\rightarrow$ $\R^{d\times n}$ satisfy  the condition: 

\begin{Assumption}
\label{Hyp:bsig}
There exists a constant $K>0$ such that
    \beqs
        |b(x, \mu) - b(x', \mu')| + | \sigma(x, \mu) - \sigma(x', \mu') | &\leq&  K \left(|x-x'| + \Wc_2(\mu, \mu')\right),
    \enqs
for all $\mu, \mu' \in \Pc_2(\R^d)$ and $x, x' \in \R^d$, and 
   \beqs
| \sigma(0, \delta_0) | + |b(0,\delta_0)|  &<&  \infty,
    \enqs
where $\delta_0 \in \Pc(\R^d)$ is the Dirac delta centered at 0.
\end{Assumption}

\begin{Remark}
\label{rem:comparison}
System \eqref{sysN} slightly differs from the one proposed in the recent literature, e.g., \cite{delattre_note_2016, bayraktar_graphon_2023, bet_weakly_2023}, in two ways: (1) it is nonlinear with respect to the empirical measure, (2) the interaction is renormalised by the exact number of neighbors in the graph.
Concerning (1), we observe that we are able to extend the classical scalar formulation
\begin{align} \label{scalargraphon} 
\d X_t^{i,N} &= \; \frac{1}{N} \sum_{j=1}^N \xi_{ij}^N \Big[  \tilde b(X_t^{i,N},X_t^{j,N}) \d t +  \tilde \sigma(X_t^{i,N},X_t^{j,N}) \d W_t^{u_i} \Big], 
\end{align}
to the case where the interaction directly depends on the local empirical measure $\nu_t^{i,N}$, recall \eqref{d:local-emp-mes}, i.e., the scalar interaction framework \eqref{scalargraphon} is a special case of our more general setting. Concerning (2), we are forced to define $\nu_t^{i,N}$ as a probability measure to be consistent with the limit formulation, see Definition \ref{d:solution-graphonsysgen}. However, this last choice is not new in the literature, see, e.g., \cite{bhamidi_weakly_2019}, and implies that there are no isolated particles, i.e. $N_i$ $>$ $0$ for all $i$ (otherwise $\nu_t^{i,N}$ would not be defined). 
\end{Remark}

We are interested in the study of the limiting  particles system when the graphon $G_N$ converges to a graphon $G$ in cut-norm when $N$ goes to infinity, see \cite{lovasz_large_2012} (and recap in the next section).  It is formally expected 
a convergence to the graphon system with a continuum of agents/particles $u$ $\in$ $I$ driven by:
\begin{align} \label{graphonsysgen}
\d X_t^u  &  = \;  b\Big(X^u_t,   \frac{1}{\int_I G(u,v) \d v}  \int_I G(u,v) \P_{X_t^v}(\d y) \d v \Big)  \d t  \\
&  \quad \quad  + \;  \sigma\Big(X^u_t,  \frac{1}{\int_I G(u,v) \d v}  \int_I G(u,v) \P_{X_t^v}(\d y) \d v \Big) \; \d W_t^u, \quad 0 \leq t \leq T.  
\end{align}

In the sequel, we shall address on the one hand the well-posedness, existence and uniqueness of a solution to  the system \eqref{graphonsysgen},  and on the other hand, the convergence (law of large numbers and propagation of chaos)  of the $N$-particle system 
\eqref{sysN}  to the graphon system \eqref{graphonsysgen}.

\section{$L^2$-formulation of graphon system}
\label{s:l2-formulation}

We now aim to formulate  the graphon system \eqref{graphonsysgen} as an equation in infinite dimension by identifying  the collection of state variables of the continuum of agents  $\{X^u: u \in I\}$ with $\boX$ $=$ $(\boX_t)_{t\in[0,T]}$ where $\boX_t$ is the mapping: 
$(\omega,u)\in \Omega\times I$ $\mapsto$ $X_t^u(\omega)$ $\in$ $\R^d$, for $t$ $\in$ $[0,T]$. However, there is a measurability issue due to the fact that $(\omega,u)$ $\mapsto$ $W^u(\omega)$, hence $X_t^u(\omega)$,  is not jointly measurable in the product space of the usual continuum product and the classical Lebesgue space $(I, \Bc_I, \lambda_I)$ on the index space $I$.  
This issue is overcome in \cite{aurell_stochastic_2022} with the notion of rich Fubini extension, previously introduced in \cite{sun_exact_2006}. Indeed, there exists a probability space $(I, \Ic, \lambda)$ extending the usual Lebesque space $(I, \Bc_I, \lambda_I(\d u) = \d u)$, and a Fubini extension
$(\Omega\times I,\Fc\boxtimes\Ic,\P\boxtimes\lambda)$ of $(\Omega\times I,\Fc\otimes\Ic,\P\otimes\lambda)$, such that $X^u_t(\omega)$ is jointly $\Fc\boxtimes\Ic$-measurable in $(\omega,u)$. We now review the key properties of $(\Omega\times I,\Fc\boxtimes\Ic,\P\boxtimes\lambda)$ and discuss its use in our framework.

\begin{Remark}
    Observe that there is no chance to obtain a continuous mapping $u \mapsto X^u$. Indeed, continuity would imply measurability with respect to the Lebesgue measure which is not possible.
\end{Remark}

\subsection{Setting and notation}

We review the key result that grants the existence of $(\Omega\times I,\Fc\boxtimes\Ic,\P\boxtimes\lambda)$ such that $\boX$ $=$ $(\boX_t)_{t\in[0,T]}$ is $\Fc\boxtimes\Ic$-measurable. 
We start by recalling the definition of essentially pairwise independent (e.p.i.) random variables indexed by $u \in I$, i.e., $\{Y^u: u \in I\}$ is a collection of e.p.i. random variables if for $\lambda$-a.e. $u\in I$ the random variable $Y^u$ is independent of $Y^v$ for $\lambda$-a.e. $v \in I$. The following theorem is taken from \cite[Theorem 1]{aurell_stochastic_2022}.

\begin{Theorem}
    \label{thm:rich-fubini}
    Let $S$ be a Polish space. There exists a probability space $(I, \Ic, \lambda)$ extending the usual Lebesgue space $(I, \Bc, \lambda_I)$, a probability space $(\Omega, \Fc, \Pc)$ and a Fubini extension $(\Omega\times I,\Fc\boxtimes\Ic,\P\boxtimes\lambda)$ of $(\Omega\times I,\Fc\otimes\Ic,\P\otimes\lambda)$ such that for any measurable mapping $\phi$ from $(I, \Ic, \lambda)$ to $\Pc(S)$, there is a $\Fc\boxtimes\Ic$-measurable process $f:\Omega \times I \to S$ such that the random variables $f_x$ are e.p.i. and $\P \circ f^{-1}_x= \phi(x)$ for $x \in I$.
\end{Theorem}

Denote by $\Cc_d$ $:=$ $C([0,T],\R^d)$ the set of continuous functions from $[0,T]$ into $\R^d$, and let 
$\phi: I \to \Pc(\Cc_d\times \R^d)$ be the mapping defined by $\phi(u) = q^u\otimes \mu^u_0$ with $q^u$ the Wiener measure on $\Cc_d$ and $\boldsymbol{\mu}_0: I \to \Pc(\R^d)$ a $\Bc_I$-measurable  (hence $\Ic$-measurable) function.  
From Theorem \ref{thm:rich-fubini}, there exists $\boldsymbol{Z}: \Omega \times I \to \Cc_d$ a $\Fc\boxtimes\Ic$-measurable process defined for $u\in I$ by $\boldsymbol{Z}^u(\cdot) = (W^u(\cdot), \zeta^u(\cdot))$ and such that
\begin{equation} \label{Fubini} 
\begin{cases}    
 (W^u, \zeta^u)_{u \in I} \mbox{ are   e.p.i. random variables}  \\
 \mbox{The law of }  (W^u, \zeta^u) \mbox{  is equal to }  q^u\otimes \mu^u_0, \mbox{  for all } u \in I.
\end{cases}
\end{equation}

We refer to \cite{aurell_stochastic_2022,carmona_probabilistic_2018} and the references therein for more information on rich Fubini extensions.

\medskip

Let $L^2_\boxtimes(\Omega\times I, \Cc_d)$ denote the space of equivalence classes of $(\Fc\boxtimes\Ic, \Bc(\Cc_d))$-measurable functions which are $\P \boxtimes \lambda$-square integrable, i.e., $\boldsymbol{\phi} \in L^2_\boxtimes(\Omega\times I, \Cc_d)$ if
\beqs
    \E^{\boxtimes} \big[ \sup_{0\leq t\leq T} | \boldsymbol{\phi}_t|^2 \big] & := & \int_{\Omega\times I} \sup_{0\leq t\leq T} | \phi_t^u(\omega)|^2 \P\boxtimes\lambda(\d \omega,\d u) \; < \; \infty. 
\enqs
We write $L_\lambda^2(I,\R^d)$ for the Hilbert space  of $\lambda$-a.e. 
equivalent classes of  $\Ic$-measurable functions $\boldsymbol{\varphi}$:  $u$ $\in$ $I$  $\mapsto$ $\varphi^u$ $\in$ $\R^d$, written in short as $\boldsymbol{\varphi}$ $=$ $(\varphi^u)_{u\in I}$, such that $\int_I |\varphi^u|^2 \lambda(\d u)$ $<$ $\infty$. We equip $L_\lambda^2(I,\R^d)$ with the standard scalar product $<\boldsymbol{\varphi},\boldsymbol{\psi}>_{_\lambda}$ $=$ $\int_I \varphi^u \cdot \psi^u \lambda(\d u)$ and the induced norm $\|\boldsymbol{\varphi}\|_{_\lambda}$ $=$ $\big( \int_I |\varphi^u|^2 \lambda(\d u) \big)^{1\over 2}$. We observe that $L^2_\lambda(I, \R^d)$ is strictly larger from the standard $L^2(I, \R^d)$ Hilbert space. We sometimes write $L^2_\lambda(I)$ $=$ $L^2_\lambda(I,\R^d)$, and $L^2(I)$ $=$ $L^2(I,\R^d)$.

\subsection{Graphons}
\label{sec:graphon}

Let  $\Gc$ be the space of bounded, symmetric and  $\lambda_I$-measurable functions on $I\times I$. A graphon is an element in $\Gc$ valued in $[0,1]$.   For $G$ $\in$ $\Gc$, its cut-norm is defined by 
\begin{align}
\Vert G\Vert_{\square} & := \; \sup_{S,S' \in \Bc(I)} \Big| \int_{S\times S'} G(u,v) \d u \d v \Big|.  
\end{align}
We denote by $\Vert G \Vert_1$ the standard $L^1$ norm, i.e., $\Vert G \Vert_1$ $=$ $\int_{I^2} G(u,v) \d u \d v$, and we shall assume 

\begin{Assumption} \label{Hyp:Gpos}
\begin{enumerate}
    \item[(i)] For $\lambda_I$-a.e. $u$ $\in$ $I$, $\Vert G(u, \cdot) \Vert_1 \; = \; \int_I G(u,v) \d v  > \; 0$, 
    \item[(ii)] $\int_I  \Vert G(u, \cdot) \Vert_1^{-1} \d u  < \; \infty$
\end{enumerate}
\end{Assumption}
or the stronger assumption 
\begin{Assumption} \label{Hyp:Gposuni}
\begin{align}
G_\infty^{-1} \; := \; \sup_{u\in I} \Vert G(u, \cdot) \Vert_1^{-1}  & < \; \infty.    
\end{align}
\end{Assumption}
Here and in the rest of the paper, the supremum with respect to $u\in I$ should be intended as an essential supremum with respect to $\lambda$, i.e., the $L^\infty$ norm on $(I, \lambda)$.

Assumption \ref{Hyp:Gposuni} requires the degree to be bounded from below. Examples of graphons satisfying this condition are the constant graphon and, more generally, the Stochastic Block model with no isolated population, but also bipartite graphons and Cayley graphons \cite{coppini_note_2022}. We believe this assumption to be necessary for the Propagation of Chaos result in view of the estimate given in Lemma \ref{lemFG}: the less connections a particle has, the less control we have between its law and the limit law in the infinite system. Assumption \ref{Hyp:Gpos} is a rather weak condition which includes graphons with vanishing degree as $G_p(u,v)=(uv)^p$ for any $0\leq p<1$. Indeed,
\[
    \Vert G(u, \cdot) \Vert_1 = \int_0^1 (uv)^p \d v = \frac {u^p}{1+p}
\]
and
\[
    \int_0^1 \Vert G(u, \cdot) \Vert_1^{-1} = (1+p) \int_0^1 u^{-p} \d u = \frac {1+p}{1-p} < \infty.
\]
Observe that both Assumption \ref{Hyp:Gpos} and \ref{Hyp:Gposuni} require no regularity on the graphon as a function from $I^2$ to $I$. It is not difficult to construct examples of irregular graphons satisfying a condition on the degree as this last one is not related to continuity.

\vspace{1mm}
 
When the context is clear, we also view  $G$ as an integral operator $G$ $:$ ${L^2(I) \to L^2(I)}$ defined by
\begin{equation}
[Gf](u) = \|G(u,\cdot)\|_1^{-1}\int_I G(u,v)f(v) \d v, \quad \lambda_I\text{-a.e. } u \in I, \quad f \in  L^2(I). 
\end{equation} 
It is not difficult to see that $G$ is a symmetric Hilbert-Schmidt operator.


For a normed space $E$ with norm defined by $|\cdot|$,  define the space  $\Mc(E) \subset \Pc_2(E)^I$ of (products of) probability measures by
\begin{align}\label{def:Mc}
    \Mc (E) :=\{ \boldsymbol{\mu}\in \Pc_2(E)^I \text{ s.t.}\quad \forall B\in\Bc(E) \quad \boldsymbol{\mu}(B): u\mapsto \mu^u(B) \text{ is } \lambda_I\text{-measurable}\\
    \text{ and s.t. } \sup_{u\in I}\int_{E} |x|^2\mu^u(\d x) < \infty \}.
\end{align}
We will simply write $\Mc$ instead of $\Mc(E)$ when the context is clear. The space $\Mc$ is a metric space when endowed with the distance defined for any $\boldsymbol\mu, \boldsymbol\nu \in \Pc_2(E)^I$ by
\begin{equation}\label{def:Mc-distance}
    \bd(\boldsymbol\mu, \boldsymbol\nu) = \sup_{u \in I} \Wc_2(\mu^u, \nu^u) = \sup_{u\in I} \left(\inf_{X^u\sim\mu^u, Y^u\sim\nu^u} \E\Big[|X^u - Y^u|^2\Big]\right)^{1/2}, 
\end{equation}
where $\Wc_2$ is the classical 2-Wasserstein distance on $\Pc_2(E)$, recall definition \eqref{def:w_2-distance}.

\vspace{1mm}

In the sequel, we will need the following lemma relying the graphon operator $G$ with the space $\Mc$. 

\begin{Lemma}
\label{lemma:graphon-operator}
Let  Assumption \ref{Hyp:Gpos}(i) hold. 
Then $G$ can be extended to the operator  
\begin{equation}
    \begin{split}
        G: \Mc(E) &\to \Mc(E)\\
        \boldsymbol{\mu} &\mapsto   [G \boldsymbol{\mu}]
    \end{split}
\end{equation}
where for $\lambda$-a.e. $u\in I$, $[G\boldsymbol{\mu}]^u$ is defined by
\begin{align}\label{def:g-operator}
    [G\boldsymbol{\mu}]^u(\d x) &:= \;  \Vert G(u, \cdot) \Vert_1^{-1} \int_I G(u,v) \mu^v(\d x) \d v \;  \in \Pc_2(E).  
\end{align} 
In particular, for any $\boldsymbol{\nu}, \boldsymbol{\eta} \in \Mc(E)$, it holds that 
\begin{align} \label{estimG}
    \sup_{u \in I} \Wc_2([G\boldsymbol{\nu}]^u, [G\boldsymbol{\eta}]^u) &\leq \;  \sup_{u \in I} \Wc_2(\nu^u, \eta^u).
\end{align}

\end{Lemma}
\begin{proof}
It is easy to see that $[G\boldsymbol{\mu}]^u$ defines a probability measure. The map $u\mapsto [G\boldsymbol\mu]^u(B)$ is measurable for any $B\in\Bc(E)$ because of Fubini's theorem. Indeed, $G(u,v)\mu^v(B)$ is jointly measurable in $(u,v)$. To see that $G(\boldsymbol\mu)\in\Mc$, it remains to prove that
\begin{align}
\int_E |x|^2 [G\bomu]^u(\d x) & < \infty, \quad \forall u \in I,  \;\;  \mbox{ and } \;\; \sup_{u\in I} \int_E |x|^2 [G\bomu]^u(\d x) \; < \;   \infty.
\end{align}
Observe that
\begin{equation*}
\begin{split}
    \int_E |x|^2 [G\bomu]^u(\d x) & = \Vert G(u, \cdot) \Vert_1^{-1} \int_I \int_E G(u,v) |x|^2 \mu^v(\d x) \d v \\
    & \le \|G(u,\cdot)\|_1^{-1} \int_I G(u,v)\d v\left(\sup_{v\in I} \int_E |x|^2 \mu^v(\d x)\right)  \; < \; C,
\end{split}
\end{equation*}
where we have used the fact that $\bomu$ $\in$ $\Mc$. The other statement follows by observing that the constant $C$ does not depend on $u$.

Finally,  to prove the continuity estimate, we consider the Kantorovitch duality of Wasserstein distance which states that (recall equation \eqref{eq:kantorovitch}) for any $\mu,\nu\in\Pc_2(E)$,
\begin{align*}
\Wc_2^2(\nu,\mu) &= \sup\Bigg\{ \int_{E} f(y)\mu(\d y) + \int_{E}g(y)\nu(\d y), \quad \text{with } f,g\in C_b(E) \text{ such that } \\
& \hspace{3cm}  f(x)+g(y) \; \le |x- y|^2 \quad\forall x,y\in E \Bigg\}. 
\end{align*}
For  $f,g$ as above, we then write for all  $\bonu$, $\boeta$ $\in$ $\Pc_2(E)^I$,  $u$ $\in$ $I$, 
\begin{align} \label{derivG}
    & \int_{E} f(y)[G\boldsymbol{\nu}]^u(\d y) + \int_{E} g(y)[G\boldsymbol{\eta}]^u(\d y)  \\
  =   &  \; \Vert G(u, \cdot) \Vert_1^{-1} \Big( \int_{E} f(y)\int_I G(u,v)\nu^v(\d y)\d v + \int_{E} g(y)\int_I G(u,v)
  \eta^v(\d y)\d v \Big)  \\
=    &\;  \Vert G(u, \cdot) \Vert_1^{-1} \int_I G(u,v) \Big( \int_{E} f(y)\nu^v(\d y) + \int_{E} g(y)\eta^v(\d y) \Big)\d v \label{interG} \\
\le     &\;  \Vert G(u, \cdot) \Vert_1^{-1} 
     \int_I G(u,v)  \d v   
     \Big[ \sup_{v\in I}\sup_{f,g}\Big( \int_{E} f(y)\nu^v(dy) + \int_{E} g(y)\eta^v(\d y) \Big) \Big]  \\
=    &\;  \Vert G(u, \cdot) \Vert_1^{-1} \int_I G(u,v)\d v \sup_{v\in I}\Wc_2^2(\nu^v,\mu^v)  \; =  \;  \sup_{v\in I}\Wc_2^2(\nu^v,\mu^v).
\end{align}
Taking $\sup_{u\in I}\sup_{f,g}$ on the left-hand side, we obtain the estimate \eqref{estimG}.
\end{proof}


\subsection{Graphon particle system in $L^2_\lambda$} 

Denote by $\boldsymbol{W}$ $:$ $(\omega,u)$ $\in$ $\Omega\times I$ $\mapsto$ $(W_t^u(\omega))_{t\in [0,T]}$ $\in$  $\Cc_d$, and 
$\boldsymbol{\zeta}$ $:$ $(\omega,u)$ $\in$ $\Omega\times I$ $\mapsto$ $\zeta^u(\omega) \in \R^d$, and $\F$ the filtration generated by $\boldsymbol W$ and $\bozeta$. 
The graphon system \eqref{graphonsysgen} can be formally written as a stochastic equation in $L_\lambda^2(I,\R^d)$: 
\begin{align} \label{dynboX}
\d \boX_t &= \; \mfb(\boX_t,(\P_{X^u_t})_{u \in I}) \d t + \mfs(\boX_t,(\P_{X^u_t})_{u \in I}) \d \boldsymbol{W}_t, \quad 0 \leq t\leq T, \; 
\boX_0 = \bozeta, 
\end{align}
where for $\bolx$ $=$ $(x^u)_{u\in I}$ $\in$ $L_\lambda^2(I,\R^d)$ and $\boldsymbol{\mu} \in \Pc_2(\R^d)^I$ the drift
\[
\mfb = (\mfb^u)_{u\in I} : L_\lambda^2(I,\R^d)\times\Pc_2(\R^d)^I \rightarrow  L_\lambda^2(I,\R^d)
\]
is defined by $\mfb^u\left( \bolx,\bomu\right) = \;  b\left(x^u, [G \boldsymbol{\mu}]^u \right)$ for $\lambda$-a.e. $u\in I$, and the diffusion
\[
\mfs = (\mfs^u)_{u\in I} : L_\lambda^2(I,\R^d)\times\Pc_2(\R^d)^I \rightarrow  L_\lambda^2(I,\R^{d\times n})
\]
is defined by $\mfs^u( \bolx,\bomu) \; = \;  \sigma \left(x^u, [G \boldsymbol{\mu}]^u \right)$ for $\lambda$-a.e. $u\in I$.

\medskip

We give the following definition
\begin{Definition}
\label{d:solution-graphonsysgen}
A solution to \eqref{dynboX} is an $\F$-progressively measurable process  $\boX\in L^2_\boxtimes(\Omega\times I, \Cc_d)$ satisfying  for  $u$ $\in$ $I$ and $\P\boxtimes\lambda$-a.e.
\begin{align}
X_t^u & = \; \zeta^{u} + \int_0^t \mfb^u(\boX_s,(\P_{X^v_s})_{v \in I}) \d s + \int_0^t \mfs^u(\boX_s,(\P_{X^v_s})_{v \in I}) \d W_s^u, \quad 
0 \leq t \leq T.
\end{align}
\end{Definition}

\begin{Remark}
The formulation of the stochastic equation \eqref{dynboX} in the Hilbert space $L_\lambda^2(I,\R^d)$ is different from the more standard one with cylindrical formulation as in \cite{cosetal21}. Here, the noise is driven by the collection of e.p.i. Brownian motions $(W^u)_{u\in I}$, which defines the process $\bW$, and a stochastic integral in the following sense: 
given an $\F$-progressively measurable process $(\boldsymbol{\phi}_t)_t \in L^2_\boxtimes(\Omega\times I,\Cc([0,T],\R^{d\times n}))$, the stochastic integral $\boldsymbol{M}_.$ $=$ $\int_0^. \boldsymbol{\phi}_t \d \bW_t$ is  the $\F$-progressively measurable process 
\begin{align}
M_.^u &= \; \int_0^. \phi_t^u \d W_t^u, \quad u \in U.
\end{align}
Notice that by Fubini extension, and isometry of stochastic integral w.r.t. Brownian motion, 
\begin{align}
\E_\boxtimes\big[ |M_T|^2 ] & = \; \int_I \E \Big| \int_0^T \phi_t^u \d W_t^u \Big|^2 \lambda(\d u) \; = \; 
\E_\boxtimes \Big[ \int_0^T | \boldsymbol{\phi}_t|^2 \d t \Big]<\infty. 
\end{align}
Moreover, since $\boldsymbol\phi$ and $\boldsymbol W$ are $\Fc\boxtimes\Ic$-measurable, then also $\boldsymbol M$ is $\Fc\boxtimes\Ic$-measurable. Thus, $\boldsymbol M \in L^2_\boxtimes(\Omega\times I, \Cc_d)$. More details are presented in Appendice \ref{secappen:L2}.  \end{Remark}

Let us now first  prove the existence and uniqueness of the solution $\boX$ to \eqref{dynboX}  in the sense of Definition \ref{d:solution-graphonsysgen}.  
Recall that $I \ni u \mapsto \mu_0^u = \Lc(\zeta^u)  \in \Pc(\R^d)$ is $\Bc_I$-measurable. We further suppose the following assumption on the initial condition:

\begin{Assumption}
\label{ass:initial-condition}
The map $I\ni u \mapsto \mu_0^u\in \Pc(\R^d)$ is $\lambda_I$-measurable, and there exists some $\epsilon>0$ s.t.
\begin{align}
\sup_{u\in I} \int_{\R^d}  |x|^{2+\epsilon} \mu_0^u(\d x) &< \;  \infty.    
\end{align}
\end{Assumption}

\begin{Proposition}
\label{p:existence-uniqueness-graphonsysgen}
Suppose Assumptions \ref{Hyp:bsig}, \ref{Hyp:Gpos}(i), and \ref{ass:initial-condition} hold. Then,  there exists a unique solution $\boX$ $:$ $(\omega,u)$ $\mapsto$ $X^u(\omega)$ to \eqref{dynboX} in the sense of Definition \ref{d:solution-graphonsysgen}, and 
for any $t$ $\in$ $[0,T]$, $(X_t^u)_{u\in I}$ is a collection of e.p.i. random variables.  
Moreover, the process $\boX$ is $\P\boxtimes\lambda$-measurable, lies in $L^2_\boxtimes(\Omega\times I, \Cc_d)$ with  
\begin{align}
\sup_{u\in I} \E \big[ \sup_{0\leq t \leq T} |X_t^u|^{2+\epsilon} \big] & < \; \infty,     
\end{align}
and $u$ $\in$ $I$ $\mapsto$ $\mu^u$ $:=$ $\P_{X_u}$ $\in$ $\Pc(\Cc_d)$ is $\Bc_I$-measurable.
\end{Proposition}
\begin{proof}
See Appendix \ref{sec:appenexis}. 
\end{proof}

\vspace{1mm}

We state some continuity properties of the graphon system. These conditions are classical when working with the particle trajectories: notably, the Lipschitz condition stated below, see Assumption \ref{hyp:lipmu0}, is necessary to have a uniform control for the Propagation of Chaos argument. We refer to \cite{bayraktar_graphon_2023} for similar assumptions and more about this choice, and to \cite{bet_weakly_2023} for weaker assumptions on the graphon but with somehow non-comparable hypothesis for the initial conditions.

\begin{Assumption} \label{hyp:uni-cont}
There exists a finite collection of intervals $\{J_i : i=1,\dots,n\}$ for some $n\in\N$ such that $\cup_{i=1}^nJ_i = I$ and for all $i\in\llbracket 1,n\rrbracket$:
\begin{enumerate}
\item the map $J_i \ni u$ $\mapsto$ $\mu_0^u$ $\in$ 
$\Pc(\R^d)$ is continuous w.r.t. $\Wc_2$,
\item for each $i, j \in \llbracket 1,n\rrbracket$, $G$ is uniformly continuous on $J_i\times J_j$. 
\end{enumerate}
\end{Assumption}

or the stronger assumption

\begin{Assumption} \label{hyp:lipmu0}
There exists a finite collection of intervals $\{J_i : i=1,\dots,n\}$ for some $n\in\N$ such that $\cup_{i=1}^nJ_i = I$ and:
\begin{enumerate}
\item for all $i\in\llbracket 1,n\rrbracket$, 
\begin{align}
\Wc_2(\mu_0^{u},\mu_0^{v}) & \leq \; \kappa|u - v|, \quad \forall u, v  \in J_i,
\end{align} 
\item $G$ is Lipschitz, i.e. 
\[
|G(u_1,v_1)-G(u_2,v_2)| \le K(|u_1-u_2| + |v_1-v_2|),\quad\forall (u_1,v_1),(u_2,v_2)\in J_i\times J_j, i,j\in\llbracket1,n\rrbracket.
\]
\end{enumerate}
\end{Assumption}

Assumption \ref{hyp:lipmu0} requires an extra blockwise regularity which is classical to obtain a precise control on the trajectory estimates. All the examples cited after Assumption \ref{Hyp:Gposuni} satisfy this condition.

We recall the definition of the truncated Wasserstein distance, for $\mu,\nu\in\Pc(\Cc^d)$, $t\in[0,T]$:
\begin{align}
    \Wc_{2,t}(\mu,\nu) := \bigg( \inf_{X\sim\mu,Y\sim\nu} \E[\sup_{0\le s\le t}|X_s-Y_s|^2] \bigg)^{1/2}.
\end{align}
We observe that it holds the following inequality:
\begin{align}
     \sup_{0\le s\le t}\Wc_2^2(\mu_s,\nu_s) \le \sup_{0\le s\le t}\E\big[|X_s-Y_s|^2\big]\le \E\big[\sup_{0\le s\le t}|X_s-Y_s|^2\big],
\end{align}
and taking the infimum over all $X\sim\mu,Y\sim\nu$ we have that:
\begin{align}
    \sup_{0\le s\le t}\Wc_2^2(\mu_s,\nu_s) \le \Wc_{2,t}^2(\mu,\nu).
\end{align}

\begin{Proposition} \label{prop-continuity}
\begin{enumerate}
\item Under Assumptions \ref{ass:initial-condition} and \ref{hyp:uni-cont}, the map $u$ $\in$ $J_i$ $\mapsto$ $\mu^u$ $\in$ $\Pc(\Cc_d)$ is continuous w.r.t. $\Wc_{2,T}$, for any $i$ $\in$ $\llbracket 1,n\rrbracket$. 
\item Furthermore, if Assumptions \ref{ass:initial-condition} and \ref{hyp:lipmu0} hold, then there exists some $\kappa$ $>$ $0$ s.t. 
$\Wc_{2,T}(\mu^u,\mu^v)$ $\leq$ $\kappa|u-v|$, for any $u,v$ $\in$ $J_i$, $i$ $\in$ $\llbracket 1,n\rrbracket$. 
\end{enumerate}
\end{Proposition}

\begin{proof}
See Appendix \ref{sec:appencont}. 
\end{proof}

\section{Convergence of the graphon mean-field system}
\label{s:lln+poc}

This section is devoted to the convergence of the $N$-particle system \eqref{sysN} towards the graphon mean-field system \eqref{dynboX} when $N$ goes to infinity.

Let us denote by $X^{u,N}$ $=$ $X^{i,N}$, $u$ $\in$ $I_i$, $i$ $\in$ $\llbracket 1,N\rrbracket$, $\boX^N$ $=$ $(X^{u,N})_{u\in I}$, $\bodelta^N$ $=$ $(\delta_{X^{u,N}})_{u \in I}$, 
and observe that the dynamics of the $N$-particle system can be rewitten as
\begin{align} \label{graphon-NL2}
\d X_t^{i,N} & = \;  \mfb^{u_i}_N(\boX^N_t,\bodelta_t^N) \d t +  
\mfs^{u_i}_N(\boX^N_t,\bodelta_t^N) \d W_t^{u_i}, \;  0 \leq t \leq T,
\end{align}
where  $\mfb^u_N\left( \bolx,\bomu\right) = \;  b\left(x^u, [G_N \boldsymbol{\mu}]^u \right)$, 
$\mfs^u_N( \bolx,\bomu) \; = \;  \sigma \left(x^u, [G_N \boldsymbol{\mu}]^u \right)$,  $u\in I$, 
for  $\bolx$ $=$ $(x^u)_{u\in I}$ $\in$ $L_\lambda^2(I,\R^d)$ and $\boldsymbol{\mu} \in \Pc_2(\R^d)^I$. 

\begin{Remark}
    The family of Brownian motions $\{W^{u_i}\}_{i\in \llbracket 1,N\rrbracket}$ in equation \eqref{graphon-NL2} is constructed by sampling the e.p.i. variables $\{W^{u}\}_{u\in I}$. In particular, $\{W^{u_i}\}_{i=1, \dots, N}$ are i.i.d. with probability 1, in the sense of
    \[
    \lambda\left(\left\{u_i, i\in \llbracket 1,N\rrbracket \text{ such that } \{W^{u_i}\}_{i\in \llbracket 1,N\rrbracket} \text{ are mutually independent}\right\}\right) = 1,
    \]
    for every $N$.
\end{Remark}

\subsection{Law of large numbers}

Denote by $\hat\bodelta$ $=$ $(\hat\delta^u)_{u\in I}$, with $\hat\delta^u$ $=$ $\delta_{X^{u_j}}$ for $u$ $\in$ $I_j$, and 
$\hat\bomu$ $=$ $(\hat\mu^u)_{u\in I}$  with  $\hat\mu^u$ $=$ $\mu^{u_j}$ for $u$ $\in$ $I_j$. 
The following result adapts to our graphon framework the classical result in \cite{fournier-guillin} about the rate of convergence in Wasserstein distance of the empirical measure of i.i.d. random variables.  

\begin{Lemma} \label{lemFG} 
Let Assumptions \ref{Hyp:bsig} and \ref{ass:initial-condition} hold, let $i \in \llbracket 1,N\rrbracket$. 
Then, there is some positive constant $C$ s.t. for all $N$ $\in$ $\N^*$, $t$ $\in$ $[0,T]$,
\begin{align} \label{estimFG1}
\E\big[ \Wc_2^2([G_N\hat\bodelta_t]^{u_i}, [G_N\hat\bomu_t]^{u_i}) \big] & \leq \; \frac{C}{\sqrt{\|G_N(u_i,\cdot)\|_1}} M_N 
\end{align}
where the rate of convergence $M_N$ depends on the dimension $d$ and the integrability parameter $\epsilon$ in 
Assumption \ref{ass:initial-condition}, namely:
\begin{align}
M_N & = \;      
    \begin{cases}
        N^{-1/2} + N^{-\epsilon/(2+\epsilon)} &\text{if $d/2<2$ and $2+\epsilon\ne 4$}\\
        N^{-1/2}\log(1+N) + N^{-\epsilon/(2+\epsilon)} &\text{if $d/2=2$ and $2+\epsilon\ne4$}\\
        N^{-2/d} + N^{-\epsilon/(2+\epsilon)} &\text{if $d/2>2$ and $2+\epsilon\ne\frac{d}{d-2}$}.
    \end{cases}
\end{align}
\end{Lemma}
\begin{proof}
Following the notation in \cite{fournier-guillin}, we introduce $B_0:=(-1,1]^d$ and, for $n\ge 1$, $B_n:=(-2^n,2^n]^d\setminus(-2^{n-1},2^{n-1}]^d $. For $l\ge 0$, $\Pc_l$ will denote the natural partition of $(-1,1]^d$ into $2^{dl}$ translations of $(-2^{-l},2^{-l}]^d$.
We start by observing that Lemma 5 and Lemma 6 in \cite{fournier-guillin} imply that, for all pairs of probability measures $\mu,\nu$ on $\R^d$:
\[
\Wc_2^2(\mu,\nu) \le K_d\underbrace {\sum_{n\ge 0}2^{2n}\sum_{l\ge 0}2^{-2l}\sum_{F\in\Pc_l}|\mu(2^nF\cap B_n)-\nu(2^nF\cap B_n)|}_{=:\alpha(\mu,\nu)},
\]
where $K_d$ is a constant depending only on the dimension $d$.
So it is sufficient to study $\E[\alpha(\mu,\nu)]$.
Now, we observe that for a Borel subset $A\subset\R^d$ we have for all $i$ $\in$ $\llbracket 1,N\rrbracket$, 
$t$ $\in$ $[0,T]$, 
\begin{align} \label{step1}
\E\big[ \big|[G_N\hat\bodelta_t]^{u_i}(A) - [G_N\hat\bomu_t]^{u_i}(A)\big| \big] 
&\le \; 
\min\bigg\{ 2[G_N\hat\bomu_t]^{u_i}(A), 
\sqrt{\frac{[G_N\hat\bomu_t]^{u_i}(A)}{N_i}} \bigg\}.
\end{align}
This comes on the one hand from the inequality 
\begin{align*}
   \E\big[ \big|[G_N\hat\bodelta_t]^{u_i}(A) - [G_N\hat\bomu_t]^{u_i}(A)\big| \big] &\le \frac{1}{N_i}\sum_{j=1}^N \xi_{ij}^N\E\big[|\delta_t^j(A)-\mu_t^{u_j}(A)|\big]\\
   &\le \frac{2}{N_i} \sum_{j=1}^N\xi_{ij}^N\mu_t^{u_j}(A) = 2[G_N\hat\mu_t]^{u_i}(A),
\end{align*}
and, on the other hand, from the fact that  
\begin{align*}
    & \E\big[ \big|[G_N\hat\bodelta_t]^{u_i}(A) - [G_N\hat\bomu_t]^{u_i}(A)\big| \big]^2 \le \E\big[ \big|[G_N\hat\bodelta_t]^{u_i}(A) - [G_N\hat\bomu_t]^{u_i}(A)\big|^2 \big]\\
    & = \frac{1}{N_i^2}\sum_{j=1}^N(\xi_{ij}^N)^2\E\big[|\delta_t^{u_j}(A)-\mu_t^{u_j}(A)|^2\big]\\
    & = \frac{1}{N_i^2}\sum_{j=1}^N(\xi_{ij}^N)^2\bigg(\E\big[\I_A(X_t^{u_j})\big] + (\mu_t^{u_j}(A))^2 - 2\mu_t^{u_j}(A)\E\big[\I_A(X_t^{u_j})\big]\bigg)\\
    & = \frac{1}{N_i^2}\sum_{j=1}^N(\xi_{ij}^N)^2\big( \mu_t^{u_j}(A) - (\mu_t^{u_j}(A))^2 \big) \\
    &\le \frac{1}{N_i^2}\sum_{j=1}^N(\xi_{ij}^N)^2\mu_t^{u_j}(A) \; \le \;  \frac{[G_N\hat\mu_t]^{u_i}(A)}{N_i},
\end{align*}
where we have used the independence of $X^{u_j}$, $j=1,\dots,N$.

Using \eqref{step1}, we can deduce that, for all $n\ge 0$, $l\ge 0$:
\begin{align}
&\sum_{F\in\Pc_l}\E\big[ \big|[G_N\hat\bodelta_t]^{u_i}(2^nF\cap B_n) - [G_N\hat\bomu_t]^{u_i}(2^nF\cap B_n)\big| \big] \nonumber \\
& \le \sum_{F\in\Pc_l} \min \bigg\{2[G_N\hat\bomu_t]^{u_i}(2^nF\cap B_n),\sqrt{\frac{[G_N\hat\bomu_t]^{u_i}(2^nf\cap B_n)}{N_i}}  \bigg\} \nonumber  \\
&\le \min \big\{ 2[G_N\hat\bomu_t]^{u_i}(B_n), 2^{dl/2}\sqrt{[G_N\hat\bomu_t]^{u_i}(B_n)/N_i}\big\}, \label{step2}
\end{align}
where in the last inequality of \eqref{step2}, we have used that $2\sum_{F\in\Pc_l}[G_N\hat\mu_t]^{u_i}(2^nF\cap B_n) = 2[G_N\hat\mu_t]^{u_i}(B_n)$, and that, using the Cauchy-Schwarz inequality and the fact that $\#(\Pc_l) = 2^{dl}$,
\[
\sum_{F\in\Pc_l}\sqrt{\frac{[G_N\hat\mu_t]^{u_i}(2^nF\cap B_n)}{N_i}} \le \sqrt{\sum_{F\in\Pc_l}\frac{[G_N\hat\mu_t]^{u_i}(2^nF\cap B_n)}{N_i}}2^{dl/2}=\sqrt{\frac{[G_N\hat\mu_t]^{u_i}(B_n)}{N_i}}2^{dl/2}.
\]
Now, under the condition that $\sup_{u\in I}\int_{\Cc^d}|x|^{2+\epsilon}\mu^u(\d x)<\infty$, 
we have that 
\begin{align*}
\int_{\Cc^d}|x|^{2+\epsilon}[G_N\hat\bomu]^{u_i}(\d x) = \frac{1}{N_i}\sum_{j=1}^N\xi_{ij}^N\int_{\Cc^d}|x|^{2+\epsilon}\mu^{u_j}(\d x) <C,
\end{align*}
with $C$ a constant that doesn't depend on $N$ or $i$. 
As in the proof of Theorem 1 in \cite{fournier-guillin}, 
without loss of generality we can suppose $C=1$, i.e.,  $\int_{\R^d}|x|^{2+\epsilon}[G_N\hat\bomu_t]^{u_i}(\d x)\le 1$, for all $N$ and $t\in[0,T]$. 
This implies that 
\begin{align}\label{step3}
[G_N\hat\bomu_t]^{u_i}(B_n)\le 2^{-(2+\epsilon)(n-1)}.
\end{align}
In fact we have that:
\begin{align*}
    1 \ge \int_{\R^d}|x|^{2+\epsilon}[G_N\hat\bomu_t]^{u_i}(\d x) \ge \int_{B_n}|x|^{2+\epsilon}[G_N\hat\bomu_t]^{u_i}(\d x) \ge [G_N\hat\bomu_s]^{u_i}(B_n)2^{(n-1)(2+\epsilon)},
\end{align*}
since $|x|^{2+\epsilon}\ge 2^{(n-1)(2+\epsilon)}$ on $B_n$. 
So, \eqref{step2} together with \eqref{step3} give:
\begin{align*}
\E\big[ \alpha([G_N\hat\bodelta_t]^{u_i}, [G_N\hat\bomu_t]^{u_i}) \big] & \le K\sum_{n\ge 0}2^{2n}\sum_{l\ge 0} 2^{-2l}\min \big\{ 2[G_N\hat\bomu_t]^{u_i}(B_n), 2^{dl/2}\sqrt{[G_N\hat\bomu_t]^{u_i}(B_n)/N_i}\big\}\\
&\le K\sum_{n\ge 0}2^{2n}\sum_{l\ge 0}2^{-2l}\min\{2^{-(2+\epsilon)(n-1)+1}, 2^{dl/2}(2^{-(2+\epsilon)(n-1)}/N_i)^{1/2}\}\\
& \le K 2^{3+\epsilon} \sum_{n\ge 0}2^{2n}\sum_{l\ge 0}2^{-2l} \min\Big\{ 2^{-(2+\epsilon)n},2^{dl/2}\sqrt{\frac{2^{-(2+\epsilon)n}}{N_i}} \Big\}.
\end{align*}
We recall now that $N_i=N\|G_N(u_i,\cdot)\|_1$, and substituting in the above inequality we get
\begin{align}
\E\big[ \alpha([G_N\hat\bodelta_t]^{u_i}, [G_N\hat\bomu_t]^{u_i}) \big] & \le K 2^{3+\epsilon} \sum_{n\ge 0}2^{2n}\sum_{l\ge 0}2^{-2l} \min\Big\{ 2^{-(2+\epsilon)n},2^{dl/2}\sqrt{\frac{2^{-(2+\epsilon)n}}{N\|G_N(u_i,\cdot)\|_1}}\Big\}\\
& \le K 2^{3+\epsilon} \sum_{n\ge 0}2^{2n}\sum_{l\ge 0}2^{-2l} \min\Big\{ \frac{2^{-(2+\epsilon)n}}{\sqrt{\|G_N(u_i,\cdot)\|_1}},2^{dl/2}\sqrt{\frac{2^{-(2+\epsilon)n}}{N\|G_N(u_i,\cdot)\|_1}}\Big\}\\
& \le \frac{K}{\sqrt{\|G_N(u_i,\cdot)\|_1}} 2^{3+\epsilon} \sum_{n\ge 0}2^{2n}\sum_{l\ge 0}2^{-2l} \min\Big\{ 2^{-(2+\epsilon)n},2^{dl/2}\sqrt{\frac{2^{-(2+\epsilon)n}}{N}}\Big\},
\end{align}
where the second inequality is obtained by observing that $\frac{1}{\sqrt{\|G_N(u_i,\cdot)\|_1}}\ge 1$.

The rest of the proof follows in the same way as in the proof of Theorem 1 in \cite{fournier-guillin}, and we get the same rate of convergence as stated in \eqref{estimFG1}.
\end{proof}

\begin{Theorem} \label{theoLLN}
Let Assumptions \ref{Hyp:bsig}, \ref{Hyp:Gpos}(ii), \ref{ass:initial-condition} and  \ref{hyp:uni-cont} hold. Suppose that $\Vert G_N - G \Vert_{\square}$ $\rightarrow$ $0$, when $N$ goes to infinity, and $X_0^{i,N}$ $=$ $\zeta^{u_i}$, $i$ $\in$ $\llbracket 1,N\rrbracket$. 
Then, 
\begin{align}
\lim_{N\rightarrow\infty} 
\frac{1}{N} \sum_{i=1}^N \E\Big[ \sup_{0 \leq t\leq T} \big| X_t^{i,N} - X_t^{u_i} \big|^2 \Big]  &= 
\lim_{N\rightarrow\infty}  \frac{1}{N} \sum_{i=1}^N  \E\Big[\int_0^t\Wc_2^2([G_N\bodelta^N_s]^{u_i},[G\bomu_s]^{u_i})\d s \Big] \; = 
\; 0. 
\end{align}
\end{Theorem}
\begin{proof} 
In the sequel, $K$ denotes a generic constant independent of $N$ (depending only on $b,\sigma$, $G$ and $T$) that may vary from line to line. 
Recall the notations: $\bomu$ $=$ $(\P_{X^u})_{u\in I}$ $\in$ $\Mc$, 
$\bodelta$ $=$ $(\delta_{X^{u}})_{u\in I}$, $\bodelta^N$ $=$ $(\delta_{X^{u,N}})_{u\in I}$.  
From the dynamics \eqref{dynboX}, \eqref{graphon-NL2}, and the Lipschitz conditions on $b,\sigma$ under Assumption \ref{Hyp:bsig}, we have for $i$ $\in$ $\llbracket 1,N\rrbracket$, $t$ $\in$ $[0,T]$, 
\begin{align*}
        \E\Big[\sup_{s\in [0,t]}|X^{i,N}_s-X^{u_i}_s|^2\Big] & \le \; 
        2\E\Big[\sup_{s\in[0,t]}\int_0^s |b(X_r^{i,N},[G_N\bodelta^N_r]^{u_i})-b(X_r^{u_i},[G\bomu_r]^{u_i})|^2\d r\Big] \\
        &  \quad  + \;  2\E\Big[\sup_{s\in[0,t]}\Big|\int_0^s(\sigma(X^{i,N}_r,[G_N\bodelta_r^N]^{u_i})
        -\sigma(X^{u_i}_r,[G\bomu_r]^{u_i}))\d W^{u_i}_r \Big|^2\Big] \\
         & \leq \;  K\E\Big[\int_0^t\sup_{r\le s}|X^{i,N}_r-X^{u_i}_r|^2 \d s\Big] 
          \; + \;  K\E\Big[\int_0^t\Wc_2^2([G_N\bodelta^N_s]^{u_i},[G\bomu_s]^{u_i})\d s\Big].
    \end{align*}
and thus, by  Gronwall lemma:
\begin{align} 
    \E\Big[\sup_{s\in[0,t]}|X_s^{i,N}-X_s^{u_i}|^2\Big] &  \le \;  
    K\E\Big[\int_0^t\Wc_2^2([G_N\bodelta^N_s]^{u_i},[G\bomu_s]^{u_i})\d s \Big] \nonumber \\
    & \le   K \underbrace{\E\Big[\int_0^t\Wc_2^2([G_N\bodelta^N_s]^{u_i},[G_N\bomu_s]^{u_i})\d s \Big]}_{(I_t^{u_i})} \nonumber \\
& \quad + K \underbrace{\int_0^t \Wc_2^2([G_N\bomu_s]^{u_i},[G\bomu_s]^{u_i}) \d s}_{(II_t^{u_i})}.  \label{Xgron1}
\end{align} 

We now claim that Assumption \ref{Hyp:Gpos}(ii) and the convergence in cut norm of $G_N$ to $G$ implies that there exists a maximum of $o(N)$ indexes $i\in\llbracket1,N\rrbracket$, i.e., a set $B_N$ of indexes defined for every $N$ with $\lim_N |B_N|/N =0$, such that, setting $A_N:=\llbracket1,N\rrbracket\setminus B_N$, we have that:
\begin{equation}\label{Hyp:almost-Gposuni}
\sup_N\sum_{i\in A_N}\frac{1}{N_i} = \sup_N\frac{1}{N}\sum_{i\in A_N}\frac{1}{\|G_N(u_i,\cdot)\|_1}<\infty.
\end{equation}
Equivalently, we look for a sequence of sets $\{\tilde B_N\}\subseteq\Bc(I)$ s.t. for each $N$, $\tilde B_N$ is a union of subintervals $I_i$ for some $i\in\llbracket 1,N\rrbracket$, $\lambda_I(\tilde B_N)\to 0$ as $N\to\infty$ and, setting $\tilde A_N:=\tilde B_N^c$, it holds
\begin{align} \label{equiv} 
\sup_N\int_{\tilde A_N}\frac{1}{\|G_N(u,\cdot)\|_1}\d u \; < \;  \infty. 
\end{align} 
To prove the claim, let us consider the sets
\begin{align}
\tilde B_N := \{ u\in I : \|G_N(u,\cdot)\|_1<\frac 12 \|G(u,\cdot)\|_1\},
\end{align}
thus
\begin{align}
\tilde A_N &= \{u\in I: \|G_N(u,\cdot)\|_1\ge\frac 12\|G(u,\cdot)\|_1\}
\; = \;  \{ u\in I: \frac{1}{\|G_N(u,\cdot)\|_1}\le 2\frac{1}{\|G(u,\cdot)\|_1}\}.
\end{align}
Then, we can easily see that:
\begin{align}
\label{hyp:sup-N-degree}
    \sup_N\int_{\tilde A_N}\frac{1}{\|G_N(u,\cdot)\|_1}\d u \; \le \;  2\sup_N\int_{\tilde A_N}\frac{1}{\|G(u,\cdot)\|}\d u \; \le \; 2\int_I\frac{1}{\|G(u,\cdot)\|_1} \; < \; \infty,
\end{align}
thanks to Assumption \ref{Hyp:Gpos}(ii). We need to prove now that $\lambda_I(\tilde B_N)\to 0$. In order to do so, let us fix an arbitrary $\epsilon>0$. Assumption \ref{Hyp:Gpos}(ii) ensures that there exist some $\bar K\in\N$ s.t. for
\[
S_{\bar K} := \{u\in I: \frac{1}{\|G(u,\cdot)\|_1}>\bar K\}
\]
it holds $\lambda_I(S_{\bar K})<\epsilon$. We now rewrite:
\begin{align}
    \lambda_I(\tilde B_N) &= \;  \lambda_I(\tilde B_N\cap S_{\bar K}) + \lambda_I(\tilde B_N\cap S_{\bar K}^c)
    \; < \;  \epsilon + \lambda_I(\tilde B_N\cap S_{\bar K}^c).
\end{align}
To prove that $\lambda_I(\tilde B_N\cap S_{\bar K}^c)\to 0$, let us consider:
\begin{align}
    \int_{\tilde B_N\cap S_{\bar K}^c}(\|G(u,\cdot)\|_1-\|G_N(u,\cdot)\|_1)\d u 
    & > \;  \int_{\tilde B_N\cap S_{\bar K}^c}\frac 12\|G(u,\cdot)\|_1\d u \\
    &\ge \;  \frac{1}{2\bar K}\lambda_I(\tilde B_N\cap S_{\bar K}^c),
\end{align}
and observe that the convergence in cut norm implies that $\int_{\tilde B_N\cap S_{\bar K}^c}(\|G(u,\cdot)\|_1-\|G_N(u,\cdot)\|_1)\d u\to 0$, thus $\lambda_I(\int_{\tilde B_N\cap S_{\bar K}^c}(\|G(u,\cdot)\|_1-\|G_N(u,\cdot)\|_1)\d u)\to 0$. This concludes the proof of  
\eqref{Hyp:almost-Gposuni}.

Standard arguments for stochastic differential equations ensure that Assumption \ref{ass:initial-condition} is propagated in time, recall, e.g., Proposition \ref{p:existence-uniqueness-graphonsysgen}. It is in particular true that
\begin{equation}
    \sup_{N\in\N} \sup_{i\in\llbracket1,N\rrbracket}\E\Big[\sup_{s\in[0,t]}|X_s^{i,N}-X_s^{u_i}|^2\Big] < \infty.
\end{equation}
By using this last remark together with \eqref{Hyp:almost-Gposuni},  we can thus estimate
\[
\frac 1N \sum_{i\in A_N} \E\Big[\sup_{s\in[0,t]}|X_s^{i,N}-X_s^{u_i}|^2\Big]
\]
instead of $\frac 1N \sum_{i=1}^N \E\Big[\sup_{s\in[0,t]}|X_s^{i,N}-X_s^{u_i}|^2\Big]$.
Indeed, more explicitly, we have that:
\begin{align}
\frac{1}{N}\sum_{i=1}^N\E\left[ \sup_{0\le t\le T}|X^{i,N}_t-X^{u_i}_t|^2 \right] &= \frac{1}{N}\sum_{i\in A_N}\E\left[ \sup_{0\le t\le T}|X^{i,N}_t-X^{u_i}_t|^2 \right] \\
&\quad + \frac{1}{N}\sum_{i\in B_N}\E\left[ \sup_{0\le t\le T}|X^{i,N}_t-X^{u_i}_t|^2 \right]\\
&= \frac{1}{N}\sum_{i\in A_N}\E\left[ \sup_{0\le t\le T}|X^{i,N}_t-X^{u_i}_t|^2 \right] + o_{_N}(1). 
\end{align}

We now proceed with the estimation of $(I_t^{u_i})$ and $(II^{u_i})$.

\medskip

\noindent {\it Step 1: Estimation of $(I_t^{u_i})$}.
Denote by $\hat\bodelta$ $=$ $(\hat\delta^u)_{u\in I}$, with $\hat\delta^u$ $=$ $\delta_{X^{u_j}}$ for $u$ $\in$ $I_j$, and 
$\hat\bomu$ $=$ $(\hat\mu^u)_{u\in I}$  with  $\hat\mu^u$ $=$ $\mu^{u_j}$ for $u$ $\in$ $I_j$. We then write 
\begin{align}
    (I_t^{u_i}) &\le \; 3 \int_0^t\E\big[\Wc_2^2([G_N\bodelta_s^N]^{u_i},[G_N\hat\bodelta_s]^{u_i})\big]\d s \nonumber \\
    &\qquad\qquad + 3 \underbrace{\int_0^T \E\big[\Wc_2^2([G_N\hat\bodelta_s]^{u_i}, [G_N\hat\bomu_s]^{u_i})\big]\d s}_{\Ec^i_N(1)} + 3 \underbrace{\int_0^T\E\big[\Wc_2^2([G_N\hat\bomu_s]^{u_i},[G_N\bomu_s]^{u_i})\big]\d s}_{\Ec^i_N(2)} \nonumber \\
    &\le \; \frac{1}{\Vert G_N(u_i,\cdot)\Vert_1}\int_0^t \frac{1}{N}\sum_{j\in A_N} \E\big[ \sup_{0\leq r\leq s} |X_r^{N,j}-X_r^{u_j}|^2\big]\d s 
    + K\big( \Ec_N^i(1) + \Ec_N^i(2)\big) + o_{_N}(1).\\
    &
    \label{inter12}
\end{align}
The inequality
\[
\E\big[\Wc_2^2([G_N\bodelta_s^N]^{u_i},[G_N\hat\bodelta_s]^{u_i})\big] \leq \frac{1}{\Vert G_N(u_i,\cdot)\Vert_1} \frac{1}{N}\sum_{j\in A_N} \E\big[ \sup_{0\leq r\leq s} |X_r^{N,j}-X_r^{u_j}|^2\big] + o_N(1)
\]
can be established by using the characterization of the Wasserstein distance together with the definition of $G_N$. Indeed,
\begin{align}
&\E\big[\Wc_2^2([G_N\bodelta_s^N]^{u_i},[G_N\hat\bodelta_s]^{u_i})\big]  \\
= & \; \E\big[\sup\big\{\frac 1{N_i} \sum_{j=1}^N G_N(u_i, u_j) [f(X^{j,N}_s) + g(X^{u_j}_s)] \; : f(x) + g(y) \leq |x-y|^2\big\} \big]\\
= &\; \frac 1{N_i} \sum_{j=1}^N G_N(u_i, u_j) \E\big[\sup\big\{[f(X^{j,N}_s) + g(X^{u_j}_s)] \; : f(x) + g(y) \leq |x-y|^2\big\} \big]\\
= & \frac{1}{\Vert G_N(u_i,\cdot)\Vert_1}\frac 1N \sum_{j=1}^N G_N(u_i, u_j) \E\big[|X_s^{N,j}-X_s^{u_j}|^2\big].
\end{align}

To estimate $\frac{1}{N}\sum_{i\in A_N}\Ec_N^i(2)$, we observe that
\begin{align*}
    \frac{1}{N}\sum_{i\in A_N}\Wc_2^2([G_N\hat\bomu_s]^{u_i},[G_N\bomu_s]^{u_i}) 
    & \le \; \frac{1}{N}\sum_{i\in A_N} \|G_N(u_i,\cdot)\|_1^{-1} \int_I\Wc_2^2(\hat\mu_s^v,\mu_s^v)\d v \; \\
    &= \;  \frac{1}{N}\sum_{i\in A_N}\frac{N}{N_i}\sum_{j=1}^N\int_{I_j}\Wc_2^2(\mu_s^{u_j},\mu_s^v)\d v\\
    & \le \; \sum_{i\in A_N}\frac{1}{N_i}\int_I\Wc_{2,T}^2(\hat\mu^{u_j},\mu^v)\d v   
    \; \\
    &= \;  \sum_{i\in A_N}\frac{1}{N_i} \sum_{j=1}^N\int_{I_j} \Wc_{2,T}^2(\mu^{u_j},\mu^v)\d v\\
    &\le K \sum_{j=1}^N\int_{I_j} \Wc_{2,T}^2(\mu^{u_j},\mu^v)\d v
\end{align*}
where we use Lemma \ref{lemma:graphon-operator}
for the first inequality,  the fact that  $\Wc_2^2(\mu_s^u,\mu_s^v)\le\Wc_{2,T}^2(\mu^u,\mu^v)$ for all $u,v$ $\in$ $[0,1]$ for the second inequality, and \eqref{Hyp:almost-Gposuni} for the third inequality. 
From the continuity property of $u$ $\mapsto$ $\mu^u$ w.r.t. $\Wc_{2,T}$ in Proposition 
\ref{prop-continuity}, we deduce that $\sum_{j=1}^N\Wc_{2,T}^2(\mu^{u_j},\mu^v)\I_{I_j}(v)$
converges to zero $\lambda$-a.e., 
when $N$ goes to infinity. 
Moreover, since 
\begin{align}
\sup_{u,v}\Wc_{2,T}^2(\mu^u,\mu^v) & \le \;  
\sup_{u,v}\E\big[ \sup_{0\leq t\leq T} |X_t^u-X_t^v|^2 \big] \; \le 4\sup_u\E\big[ \sup_{0\leq t\leq T} |X_t^u|^2\big] \; < \; \infty, 
\end{align}
we conclude by the dominated convergence theorem that $\sum_{j=1}^N\int_{I_j}\Wc_{2,T}^2(\mu^{u_j},\mu^v)\d v$ converges to $0$, as $N\to\infty$. 
This implies that $\frac{1}{N}\sum_{i\in A_N}\Ec_N^i(2)\rightarrow 0$.

Similarly, to estimate $\frac{1}{N}\sum_{i\in A_N}^N\Ec_N^i(1)$,
it is enough to apply the estimate \eqref{estimFG1} 
to get convergence to $0$ as $N$ goes to infinity. 

\bigskip


\noindent {\it Step 2: Estimate of $(II_t^{u_i})$.} 
We control the Wasserstein distance by means of a weighted total variation distance, recall \eqref{eq:w-control-tv}, i.e., 
\begin{align}
\label{total_var1}
\Wc_2^2([G_N\bomu_s]^u,[G\bomu_s]^u) & \le \;  2 \int_{\R^d} |x|^2\big| [G\bomu_s]^u-[G_N\bomu_s]^u \big|(\d x),
\end{align}
where $\big| [G\bomu_s]^u-[G_N\bomu_s]^u \big|$ is the variation of the (signed) measure $[G\bomu_s]^u-[G_N\bomu_s]^u$, that we rewrite as
\begin{align*}
    \big| [G\bomu_s]^u-[G_N\bomu_s]^u \big| &= \bigg| \int_I \bigg( \frac{G(u,v)}{\|G(u,\cdot)\|_1} - \frac{G_N(u,v)}{\|G_N(u,\cdot)\|_1} \bigg)\mu_s^v \d v \bigg| =: |M|.
\end{align*}
We know that $|M|=M^{+}+M^{-}$, where $M^+$ and $M^-$ are respectively the positive and the negative parts of $M$. In particular, there exist two measurable sets $P$ and $N$ such that:
\begin{itemize}
    \item $P\cup N = \R^d$ and $P\cap N=\emptyset$,
    \item $P$ is a positive set and $N$ is a negative set,
\end{itemize}
and for every measurable set $B$ we have $M^+(B) = M(B\cap P)$ and $M^-(B) = -M(B\cap N)$.
    It follows from \eqref{total_var1} that:
    \begin{align}
    (II_t^{u_i}) & \leq \; \Gc_N^i \; := \;  2 \int_0^T \int_I \left( \frac{G(u_i,v)}{\|G(u_i,\cdot)\|_1} - \frac{G_N(u_i,v)}{\|G_N(u_i,\cdot)\|_1} \right)\int_{\R^d}|x|^2(\I_P(x) -\I_N(x) ))\mu_s^v(\d x)\d v\d s,    \label{interII}
    \end{align}
    where we denote by $\I_B$ the indicator function of the set $B$. 
  
    \noindent {\it Step 3: Final estimates.} 
    By summing \eqref{Xgron1} over $i$ in $A_N$, dividing by $N$, and using \eqref{Xgron1}, \eqref{inter12}, \eqref{interII}, we obtain by Gronwall lemma: 
    \begin{align} \label{interXIII} 
     \frac{1}{N} \sum_{i\in A_N} \E\Big[\sup_{t \in[0,T]}|X_t^{i,N}-X_t^{u_i}|^2\Big]  & \leq 
    \underbrace{\frac KN\sum_{i\in A_N}(\Ec_N^i(1)+\Ec_N^i(2))}_{\Ec_N} + o_N(1) + K \frac{1}{N} \sum_{i\in A_N} \Gc_N^i.
    \end{align}
    By Step 1, we have that 
    $\Ec_N$ goes to zero as $N$ goes to infinity. 
    Let us now check that the last term  in \eqref{interXIII} converges to zero as $N$ goes to infinity. 
    To see this, we consider the following operator norm $\|G\|_{\infty\to 1}$:
    \begin{align}
    \|G\|_{\infty\to 1} & := \;  \sup_{\|g\|_\infty \le 1} \int_I\Big|\int_I G(u,v)g(v)\d v \Big|\d u,
    \end{align} 
    and we recall from \cite[Lemma 8.11]{lovasz_large_2012} that  if $\|G-G_N\|_{\square}\to 0$, then  $\|G-G_N\|_{\infty\to 1}\to 0$. 
    Let us observe that:
\begin{align}
    \Gc_N^i & = \int_0^T\int_I\int_{\R^d}\left( \frac{G(u_i,v)}{\|G(u_i,\cdot)\|_1}-\frac{G(u_i,v)}{\|G_N(u_i,\cdot)\|_1}+\frac{G(u_i,v)}{\|G_N(u_i,\cdot)\|_1} - \frac{G_N(u_i,v)}{\|G_N(u_i,\cdot)\|_1} \right)\\
        &\qquad\qquad |x|^2(\I_P(x)-\I_N(x))\mu_s^v(\d x)\d v\d s\\
        & = \int_0^T\int_I\int_{\R^d}\frac{G(u_i,v)(\|G_N(u_i,\cdot)\|_1-\|G(u_i,\cdot)\|_1)}{\|G(u_i,\cdot)\|_1\|G_N(u_i,\cdot)\|_1}|x|^2(\I_P(x)-\I_N(x))\mu_s^v(\d x)\d v\d s\\
        &\qquad+\int_0^T\int_I\int_{\R^d} \frac{G(u_i,v)-G_N(u_i,v)}{\|G_N(u_i,\cdot)\|_1}|x|^2(\I_P(x)-\I_N(x))\mu_s^v(\d x)\d v\d s\\
        & =: \Gc_N^i(1) + \Gc_N^i(2), 
\label{GN12}
\end{align}
Thus, 
\begin{align*}
    \frac 1 N\sum_{i\in A_N} \Gc_n^i &\le \; \frac 1 N\left|\sum_{i\in A_N}\Gc_N^i(1)\right| + \frac 1 N \left|\sum_{i\in A_N}\Gc_N^i(2)\right|.
\end{align*}
We start with the convergence of $\frac 1N\left|\sum_{i\in A_N}\Gc_N^i(2)\right|$:
\begin{align}
    &\frac 1N\left|\sum_{i\in A_N}\Gc_N^i(2)\right| \\
    &= \frac 1N\left|N\sum_{i\in A_N}\int_{I_i}\frac{1}{\Vert G_N(u,\cdot)\Vert_1}\int_I\left(G\left(\frac{\lceil Nu\rceil}{N},v\right)-G_N(u,v)\right)\underbrace{\int_0^T\int_{\R^d}|x|^2(\I_P(x)-\I_{\tilde P}(x))\mu_s^v(\d x)\d s}_{=:\beta(v)}\d v\d u \right|\\
    &\le \left| \int_{\tilde A_N}\frac{1}{\Vert G_N(u,\cdot)\Vert_1}\int_I\left(G\left(\frac{\lceil Nu\rceil}{N},v\right)-G(u,v)\right)\beta(v)\d v\d u \right|\\
    &\qquad\qquad + \left| \int_{\tilde A_N}\frac{1}{\Vert G_N(u,\cdot)\Vert_1}\int_I\left(G\left(u,v\right)-G_N(u,v)\right)\beta(v)\d v\d u \right| =: |\Gc_N(2,1)| + |\Gc_N(2,2)|.
\end{align}
For $\Gc_N(2,1)$ we write:
\begin{align}
    |\Gc_N(2,1)|&\le K\int_{\tilde A_N}\frac{1}{\Vert G_N(u,\cdot)\Vert_1}\int_I\left|G\left(\frac{\lceil Nu\rceil}{N},v\right)-G(u,v)\right|\d v\d u\\
    &\le K\int_{\tilde A_N}\frac{1}{\Vert G_N(u,\cdot)\Vert_1}\d u\sup_{u\in I}\int_I\left|G\left(\frac{\lceil Nu\rceil}{N},v\right)-G(u,v)\right|\d v\\
    &\le K\sup_{u\in I}\int_I\left|G\left(\frac{\lceil Nu\rceil}{N},v\right)-G(u,v)\right|\d v\xrightarrow[N\to\infty]{} 0,
\end{align}
where we have used \eqref{hyp:sup-N-degree} in the third inequality and the final convergence follows from the uniform continuity of $G$ in Assumption \ref{hyp:uni-cont}.
To prove the convergence of $|\Gc_N(2,2)|$, observe that, since Assumption \ref{Hyp:Gpos}(ii) holds, we can find a sequence of simple (hence bounded) functions $\{g_M\}$ such that, for each $M$, $0\le g_M(u)\le \frac{1}{\Vert G(u,\cdot)\Vert_1}$ $\forall u\in I$ and
\[
\int_I \left(\frac{1}{\Vert G(u,\cdot)\Vert_1} - g_M(u)\right)\d u < \frac{1}{M}.
\]  
We now write, for an arbitrary $M\in\N$:
\begin{align}
    |\Gc_N(2,2)| &\le K\int_{\tilde A_N}\frac{1}{\Vert G(u,\cdot)\Vert_1}\left| \int_I\left(G(u,v)-G_N(u,v)\right)\beta(v)\d v\right|\d u\\
    &\le K\int_{\tilde A_N}\left(\frac{1}{\Vert G(u,\cdot)\Vert_1}-g_M(u)\right)\left| \int_I\left(G(u,v)-G_N(u,v)\right)\beta(v)\d v\right|\d u\\
    &\quad + K\int_{\tilde A_N}g_M(u)\left| \int_I\left(G(u,v)-G_N(u,v)\right)\beta(v)\d v\right|\d u\\
    &\le K\int_{\tilde A_N}\left(\frac{1}{\Vert G(u,\cdot)\Vert_1}-g_M(u)\right)\d u + K(M)\int_I\left| \int_I\left(G(u,v)-G_N(u,v)\right)\beta(v)\d v\right|\d u\\
    &\le \frac KM + K(M)\Vert G-G_N\Vert_\square\xrightarrow[N\to\infty]{} \frac KM,
\end{align}
where in the first inequality we use the fact that $\frac{1}{\Vert G_N(u,\cdot)\Vert_1}\le \frac{2}{\Vert G(u,\cdot)\Vert_1}$ for $u\in\tilde A_N$. In the second last inequality $K(M)$ is a constant that may depend on $M$, but not on $N$.
Taking now $\lim_{M\to\infty}$, we obtain the desired result.
We now turn to the convergence of $\frac 1 N\left|\sum_{i\in A_N}\Gc_N^i(1)\right|$:
    \begin{align*}
        &\frac 1 N\left|\sum_{i\in A_N}\Gc_N^i(1)\right|\le \frac 1N\sum_{i\in A_N}|\Gc_N^i(1)|\\
        \le & \;  \frac 1N \sum_{i\in A_N}\frac{\big|\|G_N(u_i,\cdot)\|_1-\|G(u_i,\cdot)\|_1\big|}{\Vert G_N(u_i,\cdot)\Vert_1}\int_0^T\int_I\frac{G(u_i,v)}{\|G(u_i,\cdot)\|_1}\int_{\R^d}|x|^2\mu_s^v(\d x)\d v\d s\\
        \le & \;  \frac KN \sum_{i\in A_N}\frac{\big|\|G_N(u_i,\cdot)\|_1-\|G(u_i,\cdot)\|_1\big|}{\Vert G_N(u_i,\cdot)\Vert_1}\\
        = &  \frac 1N \sum_{i\in A_N}\left|N\int_{I_i}\int_I \frac{(G_N(u,v)-G(\frac{\lceil Nu\rceil}{N},v))}{\Vert G_N(u,\cdot)\Vert_1}\d v\d u\right|,
    \end{align*}
    and the rest of the proof is similar to what was done for $\frac 1N\left|\sum_{i\in A_N}\Gc^i_N(2)\right|$.

\end{proof}

\subsection{Propagation of chaos}

If the limit graphon $G$ satisfies both some degree condition (Assumption \ref{Hyp:Gposuni}), and some regularity assumption (Assumption \ref{hyp:lipmu0}), then the finite particle system can be controlled trajectory per trajectory. The rate can be computed and turns out to be optimal, i.e., the same of the mean-field behavior \cite{fournier-guillin}.

\begin{Theorem} \label{theoPOC}
Let Assumptions \ref{Hyp:bsig}, \ref{Hyp:Gposuni},  \ref{ass:initial-condition} and \ref{hyp:lipmu0} hold.  Suppose that $G_N(u_i,u_j)$ $=$ $G(u_i,u_j)$, $i,j$ $\in$ $\llbracket 1, N\rrbracket$,  and  $X^{i,N}_0$ $=$ $\zeta^{u_i}$, 
$i$ $\in$ $\llbracket 1,N\rrbracket$. 
Then,  there exists some positive constant $K>0$, depending on $d,\epsilon$, such that 
\beqs
\sup_{i \in \llbracket 1,N\rrbracket} \E \Big[ \sup_{t \in [0,T]} \big| X_t^{i,N} - X_t^{u_i} \big|^2\Big] &\leq & KM_N. 
\enqs
\end{Theorem}
\begin{proof}
As in  the proof of Theorem \ref{theoLLN}, see \eqref{Xgron1},  \eqref{inter12} and \eqref{interII},  we have for all $i$ $\in$ $\llbracket 1,N\rrbracket$
\begin{equation} \label{POCinter1}
\begin{split}
\E\Big[\sup_{s\in[0,t]}|X_s^{i,N}-X_s^{u_i} |^2\Big] \le & K \int_0^t \sup_{i \in \llbracket 1,N\rrbracket } 
\E\Big[\sup_{r\in[0,s]}|X_r^{i,N}-X_r^{u_i} |^2\Big] \d s\\
&+ K \Ec_N(1) + K \Ec_N(2)  + K\Gc_N,  
\end{split}
\end{equation} 
where $\Ec_N(1) := \max_i \Ec_N^i(1)$, $\Ec_N(2) := \max_i \Ec_N^i(2)$ and $\Gc_N := \max_i\Gc_N^i$.
Since $G_N(u_i, u_j) = G(u_i, u_j)$ and $\sup_{u\in I} \Vert G(u, \cdot) \Vert_1^{-1} < \infty$ from Assumption \ref{Hyp:Gposuni}, we can assume that there exists a constant $K$ such that for $N$ large enough and for all $i\in \llbracket1, \dots, N\rrbracket$, $\Vert G_N(u_i, \cdot) \Vert_1 N = N_i\ge KN$.
Thus, from the estimation \eqref{estimFG1} in Lemma \ref{lemFG}, 
we have $\Ec_N(1)$ $\leq$ $KM_N$. 
Under Assumption \ref{hyp:lipmu0}, we have from Proposition \ref{prop-continuity} that $\Ec_N(2)\leq K/{N^2}$.
On the other hand, we have that:
\begin{align*}
    \Gc_N^i & \leq \;  
    K \int_I \bigg| \frac{G(u_i,v)}{\|G(u_i,\cdot)\|_1} - \frac{G_N(u_i,v)}{\|G_N(u_i,\cdot)\|_1} \bigg|\d v \\
    & = \; 
    K \int_I \frac{1}{\|G(u_i,\cdot)\|_1}|G(u_i,v)- G_N(u_i,v)|\d v \\
    &\qquad\qquad +  
    K \int_I \frac{G_N(u_i,v)}{\|G(u_i,\cdot)\|_1\|G_N(u_i,\cdot)\|_1} \big|\|G_N(u_i,\cdot)\|_1 - \|G(u_i,\cdot)\|_1\big|\d v \\
    & = :  \Gc^i_N(1) + \Gc^i_N(2).
\end{align*}
Now, since $G_N(u_i,v) = \sum_{j,k=1}^N G(u_j,u_k)1_{I_j}(u_i)1_{I_k}(v)$, we can write 
\begin{align*}
\Gc_N^i(1) & := \; K \int_I   \frac{1}{\|G(u_i,\cdot)\|_1}| G(u_i,v) - G_N(u_i,v) | \d v \\
& = \;  \frac{K}{\|G(u_i,\cdot)\|_1} \int_I 
\bigg| \sum_{j,k=1}^N (G(u_i,v) - G\big(u_j,u_k\big))1_{I_j}(u_i)1_{I_k}(v) \bigg| \d v \\
    & \le \;  \frac{K}{\|G(u_i,\cdot)\|_1} \int_I \sum_{j,k=1}^N\bigg| G(u_i,v)-G\big(u_j,u_k\big) \bigg|1_{I_j}(u_i)
    1_{I_k}(v) \d v  \\
    & \le \;  \frac{K}{\|G(u_i,\cdot)\|_1} \int_I \sum_{j,k}^N \frac{1}{N}1_{I_j}(u_i)1_{I_k}(v) \d v 
    \; \le \;  \frac{K}{N},
\end{align*}
where we use in the second inequality the Lipschitz condition on $G$ in Assumption \ref{hyp:lipmu0}, and the fact that $\sup_{u\in I}\|G(u,\cdot)\|_1^{-1}<\infty$ in Assumption \ref{Hyp:Gposuni}. On the other hand,  we have 
\begin{align*}
\Gc_N^i(2) &:= \; K \int_I \frac{G_N(u_i,v)}{\|G(u_i,\cdot)\|_1\|G_N(u_i,\cdot)\|_1} \big|\|G_N(u_i,\cdot)\|_1 - \|G(u_i,\cdot)\|_1\big|\d v \\
& = K \int_I \frac{G_N(u_i,v)}{\|G(u_i,\cdot)\|_1\|G_N(u_i,\cdot)\|_1}  
\bigg| \int_I (G(u_i,w) - G_N(u_i,w))\d w \bigg| \d v \\
& \le \;  K \int_I \frac{G_N(u_i,v)}{\|G(u_i,\cdot)\|_1\|G_N(u_i,\cdot)\|_1} 
\sum_{j=1}^N \int_{I_j}|G(u_i,w) - G(u_i,u_j)|\d v \\
& \le \;   K \int_I \frac{G_N(u_i,v)}{\|G(u_i,\cdot)\|_1\|G_N(u_i,\cdot)\|_1} 
\sum_{j=1}^N\int_{I_j}\frac{1}{N} \d v \; \leq  \;  \frac{K}{N}.
\end{align*}
We deduce that $\Gc_N^i$ $\leq$ $K/N$ for all $i$ $\in$ $\llbracket 1,N\rrbracket$, and we conclude with \eqref{POCinter1} by Gronwall lemma.

\end{proof}

\appendix 

\section{Appendix} \label{secappen:L2}

\subsection{Stochastic integral in the Fubini extension}

  Let $\Phi\in L^2_\boxtimes(\Omega\times I; \Cc_d )$. 
  We define the stochastic integral as the following stochastic process for $(\omega,u) \in \Omega\times I$
\[
I_t^\Phi(\omega,u) := \Bigg( \int_0^t \Phi^u_s \d W^u_s\Bigg)(\omega),
\]
where $\int_0^t\Phi_t^u \d W^u_t$ is the classical integral with respect to the Brownian motion $W^u$.
In the following we will consider the filtration $\F$ generated by $W$, i.e. $\Fc_t := \sigma(W_s, s\le t)$.

Let us prove that $I^\Phi \in L^2_\boxtimes(\Omega\times I; \Cc_d)$. 
We start by checking the measurability. Since $\Phi\in L^2_\boxtimes(\Omega\times I)$, we can find a sequence of simple processes $\{\Phi^n\}$ approximating it in the sense that 
\[
 \int_I\E\bigg[\int_0^T |\Phi^{n,u}_t - \Phi^u_t|^2dt\bigg]\lambda(\d u)\to 0.
\]
A simple process takes the form 
\[
\Phi^n_t(\omega, u) = \sum_{i=0}^{n-1} \phi^i(\omega,u)\I_{[t_i,t_i+1)}(t),
\] 
where $\phi^i$ is $\F_{t_i}$-measurable and is $L^2_\boxtimes(\Omega\times I)$. The integral of $\Phi^n$ is then defined by
\[
I^{\Phi^n}_t(\omega,u) := \sum_{i=0}^{n-1} \phi^i(\omega,u)(W_{t_{i+1}}(\omega,u)-W_{t_i}).
\]
Obviously, $I^{\Phi^n}_t$ is $\Fc\boxtimes\Ic$-measurable. Because of the $L^2$ convergence, there exists a null set $\Nc \in\Ic$ such that for all $u\in I\setminus \Nc$ it holds that (up to a subsequence)
\[
\E\bigg[ \int_0^T |\Phi_t^{n,u}-\Phi_t^u|^2 \d t \bigg]\to 0
\]
and thus
\[
\E\Bigg[\bigg|{\int_0^t \phi^{n,u}_s \d W^u_s - \int_0^t \phi_s^u \d W^u_s}\bigg|^2 \Bigg] \to 0.
\]
Following the approach of \cite{yor}, for $u\in \Nc^c$ we define the following functions:
\begin{align*}
    &n_{0}(u) := 1\\
    &n_{k}(u) := \inf \Bigg\{ m>n_{k-1}(u) \text{ s.t. } \sup_{p,q\ge m} \P\Big( \Big| \int_0^t \Phi_s^{p,u} \d W^u_s - \int_0^t \Phi^{q,u}_s \d W^u_s \Big|> 2^{-k} \Big)< 2^{-k} \Bigg\}.
\end{align*}
Observe that $n_{k}$ is $\Ic$-measurable for all $k\in\N$. 
Thus, we set 
\[
Y_k^u(\omega) := \int_0^t\Phi^{n_{k}(u),u}_s \d W^u_s,
\]
and observe that $Y$ is $\Fc\boxtimes\Ic$-measurable since the composition of measurable functions. It is easy to see that for all $u\in \Nc^c$, $Y^u_k$ converges $\P$-a.s. The stochastic integral is defined for $(\omega, u)$ as
\[
\Bigg( \int_0^t \Phi^u_s \d W^u_s \Bigg)(\omega) := 
\begin{cases}
    0   &\text{if $(\omega,u)\in A^c\cup(\Omega\times\Nc)$}\\
    \lim_{k\to\infty} Y_k^u(\omega)  &\text{if $(\omega,u)\in A\setminus (\Omega\times\Nc)$}
\end{cases}
\]
where $A:=\{ (\omega,u) \text{ s.t } \lim_{k\to\infty}Y^u_k(\omega) \text{ exists and is finite } \}$.

We show that $A^c\in\Fc\boxtimes\Ic$. In fact
\begin{align*}
A^c &= \{ (\omega,u) \text{ s.t. } \exists\epsilon>0, \forall K\in\N, \exists p,q\in\N, p,q\ge K : | Y^u_p(\omega) - Y^u_q(\omega) | >\epsilon \}\\
&= \{ (\omega,u) \text{ s.t. } \exists\epsilon>0, \forall K\in\N, \exists p,q\in\N, p,q\ge K : \\
&\qquad\qquad\qquad\qquad\qquad\qquad\qquad\qquad\qquad\qquad \Big| \int_0^t\Phi^{n_{p}(u),u}_s \d W^u_s - \int_0^t\Phi^{n_{q}(u),u}_s \d W^u_s\Big| >\epsilon \}\\
&= \bigcup_{\epsilon\in\Q_+}\bigcap_{K\in\N}\bigcup_{p,q\ge K} \Big\{ \Big| \int_0^t\Phi^{n_{p}(u),u}_s \d W^u_s - \int_0^t\Phi^{n_{q}(u),u}_s \d W^u_s\Big| >\epsilon \Big\},
\end{align*}
and the sets $\Big\{ \Big| \int_0^t\Phi^{n_{p}(u),u}_s \d W^u_s - \int_0^t\Phi^{n_{q}(u),u}_s \d W^u_s\Big| >\epsilon \Big\}\in\Fc\boxtimes\Ic$. We can then deduce that $\Bigg( \int_0^t \Phi^u_s \d W^u_s \Bigg)(\omega)$ is $\Fc\boxtimes\Ic$-measurable.

The integrability property is ensured by observing that
\[
\E^\boxtimes[\sup_{t\le T}|I^\Phi_t|^2] = \int_I\E\Bigg[ \sup_{t\le T}\Big| \int_0^t\Phi^u_s \d W^u_s \Big|^2 \Bigg]\lambda(du) \le \int_I C\E\Bigg[ \int_0^T |\Phi^u_s|^2 \d s \Bigg]\lambda(\d u) < \infty,
\]
where we have used the Burkholder-Davis-Gundy inequality and where $C$ is a constant depending only on the exponent $2$ (and not on $u\in I$).

\subsection{Proof of Proposition \ref{p:existence-uniqueness-graphonsysgen}} \label{sec:appenexis}

The proof of Proposition \ref{p:existence-uniqueness-graphonsysgen} is somehow standard: originally presented in \cite{sznitman_topics_1991}, it is based on a fixed point argument which yields existence (and uniqueness) of a solution. We closely follow the proof given in \cite{bayraktar_graphon_2023} without giving all the details. Consider the mapping 
\begin{equation}
    \begin{split}
        \Phi: \Mc(\Cc_d) &\to \Mc(\Cc_d)\\
        \boldsymbol{\nu} &\mapsto \Phi(\boldsymbol{\nu}) := \Lc(\boX^{\boldsymbol\nu}),
    \end{split}
\end{equation}
where $\Lc(\boX^{\boldsymbol{\nu}}): I \to \Pc_2(\Cc_d), u\mapsto \Lc(X^{\boldsymbol\nu,u})$ and $\Lc(X^{\boldsymbol\nu,u})$ is the law of the solution of the system (with frozen probability measures $\boldsymbol\nu$)  
\begin{align}
\label{freeze-dynboX}
    \d \boX^{\boldsymbol{\nu}}_t &= \; \mfb(\boX^{\boldsymbol{\nu}}_t,\boldsymbol{\nu}_t) \d t + \mfs(\boX^{\boldsymbol{\nu}}_t,\boldsymbol{\nu}_t) \d \boldsymbol{W}_t, \quad 0 \leq t\leq T, \; 
    \boX^{\boldsymbol{\nu}}_0 = \bozeta.
\end{align}
Equation \eqref{freeze-dynboX} can be rewritten as a classical system of SDEs on $\R^d$
\begin{equation} \label{freeze-graphonsysgen}
\begin{cases}
\d X_t^{u,\boldsymbol{\nu}} &= \;  b\big(X^{u,\boldsymbol{\nu}}_t, [G\boldsymbol{\nu}_t]^u \big)  \d t \; + \;  \sigma\big(X^{u,\boldsymbol{\nu}}_t, [G\boldsymbol{\nu}_t]^u \big) \; \d W_t^u, \\
X_0^{u,\boldsymbol{\nu}} & = \zeta^u, \quad \text{for $\lambda$-a.e. } u \in I, 
\end{cases}
\end{equation}
which has progressively measurable, Lipschitz continuous and bounded coefficients. Given Assumption \ref{Hyp:bsig}, it is well-known that there is a unique time-continuous $\F$-adapted solution $(X^{u, \boldsymbol{\nu}})_{u \in I}\in L^2_\boxtimes(\Omega\times I; \Cc_d)$. 
It is then not difficult by following the same arguments as in \cite{bayraktar_graphon_2023} to see that the map $\Phi$ is well defined, i.e. $\Phi(\boldsymbol\nu)\in\Mc(\Cc_d)$.


Now, let us  prove that $\Phi$ is a contraction with respect to the metric $\bd$, recall \eqref{def:Mc-distance}. Following a classical contraction argument (\cite[Proof of Lemma 1.3]{sznitman_topics_1991}), we need to prove that there exists a constant $C>0$ such that
\[
\sup_{u\in I}\Wc^2_{2,t}(\Phi(\boldsymbol\mu)^u,\Phi(\boldsymbol\nu)^u) \le C\int_0^t \sup_{u\in I}\Wc^2_{2,r}(\mu^u,\nu^u) \d r. 
\]
Since $\Wc^2_{2,t}(\Phi(\boldsymbol\mu)^u,\Phi(\boldsymbol\nu)^u) \le \E\Big[\sup_{0\le s\le t} |X^{\boldsymbol\mu,u}_t - X^{\boldsymbol\nu, u}_t|^2\Big]$, it suffices to control the corresponding trajectories
\begin{align*}
    \E\bigg[ \sup_{s\le t} |X^{\boldsymbol\mu, u}_s - X^{\boldsymbol\nu, u}_s|^2 \bigg] &\le 2\underbrace{\E\bigg[ \int_0^t |b(X^{\boldsymbol\mu, u}_r, [G\boldsymbol\mu_r]^u) - b(X^{\boldsymbol\nu, u}_r, [G\boldsymbol\nu_r]^u)|^2 \d r\bigg]}_{\mathbf I} +\\
    &\quad + 2\underbrace{\E\bigg[ \sup_{s\le t} \bigg| \int_0^s (\sigma(X^{\boldsymbol\mu, u}_r, [G\boldsymbol\mu_r]^u) - \sigma(X^{\boldsymbol\nu, u}_r, [G\boldsymbol\nu_r]^u)) \d W^u_r \bigg|^2 \bigg]}_{\mathbf{II}}.
\end{align*}
From the Lipschitz properties of the coefficients and by Lemma \ref{lemma:graphon-operator}, we have
\begin{align*}
    \mathbf{I} &\le L \E\bigg[ \int_0^t (|X^{\boldsymbol\mu, u}_r - X^{\boldsymbol\nu,u}_r|^2 + \sup_{u\in I}\Wc_2^2([G\boldsymbol\mu_r]^u,[G\boldsymbol\nu_r]^u)) \d r \bigg]\\
    &\le L\E\bigg[ \int_0^t (|X^{\boldsymbol\mu, u}_r - X^{\boldsymbol\nu,u}_r|^2 + \sup_{u\in I}\Wc_2^2(\boldsymbol\mu_r^u,\boldsymbol\nu^u_r)) \d r \bigg],
\end{align*}
and similarly
\begin{align*}
    \mathbf{II} &\le \E\bigg[ \int_0^t |\sigma(X^{\boldsymbol\mu, u}_r, [G\boldsymbol\mu_r]^u) - \sigma(X^{\boldsymbol\nu, u}_r, [G\boldsymbol\nu_r]^u)|^2 \d r \bigg]\\
    &\le L \E\bigg[ \int_0^t (|X^{\boldsymbol\mu, u}_r - X^{\boldsymbol\nu,u}_r|^2 + \sup_{u\in I}\Wc_2^2([G\boldsymbol\mu_r]^u,[G\boldsymbol\nu_r]^u)) \d r \bigg]\\
    &\le L\E\bigg[ \int_0^t (|X^{\boldsymbol\mu, u}_r - X^{\boldsymbol\nu,u}_r|^2 + \sup_{u\in I}\Wc_2^2(\boldsymbol\mu_r^u,\boldsymbol\nu^u_r)) \d r \bigg].
\end{align*}
Fubini's theorem yields that
\begin{align*}
\E\Big[\sup_{0\le s\le t} |X^{\boldsymbol\mu,u}_t - X^{\boldsymbol\nu, u}_t|^2\Big] &\le C\int_0^t \E\big[|X^{\boldsymbol\mu, u}_r - X^{\boldsymbol\nu,u}_r|^2\big]dr + C\int_0^t \sup_{u\in I}\Wc_2^2(\boldsymbol\mu_r^u,\boldsymbol\nu^u_r) \d r\\
&\le C\int_0^t \E\big[\sup_{r'\le r}|X^{\boldsymbol\mu, u}_{r'} - X^{\boldsymbol\nu,u}_{r'}|^2\big]dr + C\int_0^t \sup_{u\in I}\Wc_2^2(\boldsymbol\mu_r^u,\boldsymbol\nu^u_r) \d r,
\end{align*}
and by Gronwall inequality we obtain
\[
\E\Big[\sup_{0\le s\le t} |X^{\boldsymbol\mu,u}_t - X^{\boldsymbol\nu, u}_t|^2\Big] \le C\int_0^t \sup_{u\in I}\Wc_2^2(\boldsymbol\mu_r^u,\boldsymbol\nu^u_r) \d r.
\]
As the estimate of $\Wc^2_{2,t}(\Phi(\boldsymbol\mu)^u,\Phi(\boldsymbol\nu)^u)$ is uniform in $u$, we conclude
\[
\sup_{u\in I}\Wc^2_{2,t}(\Phi(\boldsymbol\mu)^u,\Phi(\boldsymbol\nu)^u) \le C\int_0^t \sup_{u\in I}\Wc^2_{2,r}(\mu^u,\nu^u) \d r.
\]
The proof of the moment estimate follows from the Lipschitz property of the coefficients and from Gronwall inequality. Indeed, let us fix $t\in[0,T]$ and consider:
\begin{align*}
    \sup_{u\in I}\E\big[ \sup_{0\le s\le t}|X_s^u|^{2+\epsilon} \big]& \le \sup_uK\E\big[ \big(|X^u_0|^{2+\epsilon} +\int_0^t |b(X^u_r,[G\bomu_r]^u)|^{2+\epsilon}\d r \int_0^t|\sigma(X^u_r,[G\bomu_r]^u)|^{2 +\epsilon}\d r\big)\big]\\
    & \le \sup_u K\E\big[|X^u_0|^{2 +\epsilon}+ \int_0^t(|b(X^u_r,[G\bomu_r]^u)-b(0,[G\bodelta_0]^u)|^{2+\epsilon}+|b(0,[G\bodelta_0]^u)|^{2+\epsilon})\d r\\
    &\qquad + \int_0^t(|\sigma(X^u_r,[G\bomu_r]^u)-\sigma(0,[G\bodelta_0]^u)|^{2+\epsilon}+|\sigma(0,[G\bodelta_0]^u)|^{2+\epsilon})\d r \big],
\end{align*}
where $\bodelta_0^u = \delta_0$ for all $u\in I$. Note that we also have $[G\bodelta_0]^u = \|G(u,\cdot)\|_1^{-1}\int_IG(u,v)\delta_0\d v = \delta_0$. Thus, thanks to the Lipschitz properties of $b$ and $\sigma$:
\begin{align*}
    \sup_{u\in I}\E\big[ \sup_{0\le s\le t}|X_s^u|^{2+\epsilon} \big] & \le \sup_uK\E\big[ 1 + |X^u_0|^{2+\epsilon} + \int_0^t|X^u_r|^{2+\epsilon}\d r\big] + \sup_u\int_0^t\Wc_2^{2+\epsilon}([G\bomu_r]^u,[G\bodelta_0]^u)\d r\\
    & \le \sup_uK\E\big[ 1 + |X^u_0|^{2+\epsilon} + \int_0^t|X^u_r|^{2+\epsilon}\d r \big] +  \int_0^t\sup_u\Wc_2^{2+\epsilon}(\mu_r^u,\delta_0^u)\d r\\
    & \le \sup_uK\E\big[ 1+|X^u_0|^{2+\epsilon}\big] + \int_0^t\sup_u\E\big[|X^u_r|^{2+\epsilon}\big]\d r\\
    & \le \sup_u K\E\big[ 1+|X^u_0|^{2+\epsilon}\big] + \int_0^t\sup_u\E\big[\sup_{s\le r}|X^u_s|^{2+\epsilon}\d r.
\end{align*}
Thus, from Gronwall inequality we obtain
\begin{align*}
    \sup_{u\in I}\E\big[\sup_{0\le s\le t}|X_s^u|^{2+\epsilon}\big] \le K(1 + \sup_{u\in I}\E[|X^u_0|^{2+\epsilon}])<\infty.
\end{align*}

\subsection{Proof of Proposition \ref{prop-continuity}.}
\label{sec:appencont}

We can argue as in \cite{bayraktar_graphon_2023} (proof of Theorem $2.1$), and consider the following:
\begin{align*}
    &\tilde X^u_t = \tilde X^u_0 + \int_0^t b(\tilde X^u_s, [G\tilde\bomu_s]^u) \d s + \int_0^t \sigma(\tilde X^u_s, [G\tilde\bomu_s]^u) \d W_s, \qquad \tilde\bomu_s^u = \Lc(\tilde X^u_s),\\
    & \tilde X^v_t = \tilde X^v_0 + \int_0^t b(\tilde X^v_s, [G\tilde\bomu_s]^v) \d s + \int_0^t \sigma(\tilde X^v_s, [G\tilde\bomu_s]^v) \d W_s, \qquad \tilde\bomu_s^u = \Lc(\tilde X^u_s),
\end{align*}
where $W$ is a common Brownian motion independent of $\tilde X^u_0,\tilde X^v_0$, and $\Lc(\tilde X^u_0)=\mu^u_0$, $\Lc(\tilde X^v_0)=\mu^v_0$. From Proposition \ref{p:existence-uniqueness-graphonsysgen} we have that $\Lc(\tilde X^u) = \mu^u$ and $\Lc(\tilde X^v)=\mu^v$. 

Using the Lipschitz properties of $b,\sigma$ we have that: 
\begin{align*}
    &\E\big[ \sup_{s\le t} |\tilde X_s^u-\tilde X_s^v|^2 \big] \le\\
    & \qquad\le K\E\big[ |\tilde X^u_0-\tilde X^v_0|^2 \big] + K\int_0^t\E\big[ \sup_{r\le s}|\tilde X^u_r-\tilde X^v_r|^2 \big]\d s + K\int_0^t \Wc_2^2([G\boldsymbol\mu_s]^u,[G\boldsymbol\mu_s]^v)\d s,
\end{align*}
and by Gronwall's inequality, we get
\[
\E\big[ \sup_{s\le t} |\tilde X_s^u-\tilde X_s^v|^2 \big] \le K\E\big[ |\tilde X^u_0-\tilde X^v_0|^2 \big] +  K\int_0^t \Wc_2^2([G\boldsymbol\mu_s]^u,[G\boldsymbol\mu_s]^v)\d s.
\]
Now, from Theorem $6.15$ in \cite{villani_optimal_2009} we argue that:
\begin{align*}
\Wc_2([G\boldsymbol\mu_s]^u,[G\boldsymbol\mu_s]^v) \le&
2\int_{\R^d} |x|^2\big| [G\boldsymbol\mu_s]^u-[G\boldsymbol\mu_s]^v \big|(\d x)\\
=& 2\int_I \bigg| \frac{G(u,w)}{\|G(u,\cdot)\|_1} - \frac{G(v,w)}{\|G(v,\cdot)\|_1}\bigg| \int_{\R^d} |x|^2\mu_s^w(\d x)\d w.
\end{align*}
Taking the infimum over all random variables $\tilde X^u_0,\tilde X^v_0$ such that $\Lc(\tilde X^u_0)=\mu^u_0$ and $\Lc(\tilde X^v_0)=\mu^v_0$, we obtain:
\[
\Wc_{2,t}^2(\mu^u,\mu^v) \le K\Wc_2^2(\mu^u_0,\mu^v_0) + \int_I \bigg| \frac{G(u,w)}{\|G(u,\cdot)\|_1} - \frac{G(v,w)}{\|G(v,\cdot)\|_1}\bigg| \int_{\R^d} |x|^2\mu_s^w(\d x)\d w.
\]
In an analogous way as in the proof of Theorem \ref{theoLLN}, we obtain $(1)$ and $(2)$, using respectively Assumption \ref{hyp:uni-cont} and \ref{hyp:lipmu0}.

\section*{Acknowledgements}
The authors would like to thank Michele Coghi and the two anonymous referees for helpful comments and remarks.

\bibliographystyle{plain}
\bibliography{bibliography-2}

\end{document}